\documentclass[12pt]{amsart}

\usepackage[colorlinks=true, pdfstartview=FitV, linkcolor=blue, citecolor=blue, urlcolor=blue, breaklinks=true]{hyperref}
\usepackage{amsmath,amsfonts,amssymb,amsthm,amscd,comment,paralist,stmaryrd,etoolbox,mathtools,mathrsfs}
\usepackage[usenames,dvipsnames]{color}

\usepackage[all]{xy}
\usepackage{tikz}
\usetikzlibrary{arrows}
\usetikzlibrary{calc}
\usepackage{pgfplots}

\newcommand\C{\mathbb{C}}
\newcommand\Z{\mathbb{Z}}

\newcommand\N{\mathbb{N}}
\newcommand\F{\mathbb{F}}
\newcommand\kk{\Bbbk}
\newcommand{\q}{\mathbf{q}}
\newcommand{\vv}{\mathbf{v}}
\newcommand\id{\mathrm{Id}}
\newcommand{\sB}{\mathscr{B}}

\newcommand\cF{\mathcal{F}}

\newcommand\blambda{{\boldsymbol\lambda}}

\newcommand\txs{\textstyle}
\newcommand\dps{\displaystyle}
\newcommand{\df}[2]{\displaystyle{\frac{#1}{#2}}}

\newcommand\GL{\operatorname{GL}}
\newcommand\MU{\mathcal{MU}}
\newcommand\uvsl{U_\vv(\mathfrak{sl}_2)}

\newcommand{\rinto}{\hookrightarrow}

\newcommand{\ronto}{\twoheadrightarrow}
\newcommand{\lonto}{\twoheadleftarrow}
\newcommand\supcod[1]{\stackrel{#1}{\supset}}
\newcommand\subcod[1]{\stackrel{#1}{\subset}}
\newcommand{\ul}{\underline}

\newcommand\enc[1]{%
  \tikz[baseline=(X.base)] 
    \node (X) [draw, shape=circle, inner sep=0] {$ #1 $};}

\DeclareMathOperator{\im}{im} 
\DeclareMathOperator{\Hom}{Hom}

\DeclareMathOperator{\End}{End}
\DeclareMathOperator{\Span}{Span}

\DeclareMathOperator{\ro}{ro}
\DeclareMathOperator{\co}{co}
\DeclareMathOperator{\inv}{inv}
\newtheorem{theo}{Theorem}[section]
\newtheorem{prop}[theo]{Proposition}
\newtheorem{lem}[theo]{Lemma}
\newtheorem{cor}[theo]{Corollary}

\newtheorem{conj}[theo]{Conjecture}

\theoremstyle{definition}
\newtheorem{defin}[theo]{Definition}
\newtheorem{rem}[theo]{Remark}
\newtheorem{exa}[theo]{Example}

\numberwithin{equation}{section}
\allowdisplaybreaks

\begin{document}
%

\title{Mirabolic quantum $\mathfrak{sl}_2$}

\author{Daniele Rosso}
\address{D.~Rosso: Department of Mathematics, University of California Riverside}
\urladdr{\url{http://math.ucr.edu/~rosso/}}
\email{rosso@math.ucr.edu}

\begin{abstract}
The quantum enveloping algebra of $\mathfrak{sl}_n$ (and the quantum Schur algebras) was constructed by Beilinson-Lusztig-MacPherson as the convolution algebra of $GL_d$-invariant functions over the space of pairs of partial $n$-step flags over a finite field. In this paper we expand the construction to the mirabolic setting of triples of two partial flags and a vector, and examine the resulting convolution algebra. In the case of $n=2$, we classify the finite dimensional irreducible representations of the mirabolic quantum algebra and we prove that the category of such representations is semisimple. Finally, we describe a mirabolic version of the quantum Schur-Weyl duality, which involves the mirabolic Hecke algebra.
\end{abstract}

\subjclass[2010]{Primary 17B37; Secondary 20G43, 17B10}
\keywords{quantum enveloping algebra, convolution, flag variety, Schur-Weyl duality.}
\date{\today}


\maketitle
\thispagestyle{empty}

\tableofcontents

%
\section{Introduction}
%

\subsection{}In 1990, Beilinson, Lusztig and MacPherson (\cite{BLM}) gave a geometric realization of the quantum enveloping algebra of $\mathfrak{sl}_n$, and of the quantum Schur algebras. They used a convolution product on the variety of pairs of $n$-step partial flags in a vector space of dimension $d$ over a finite field to obtain the quantum Schur algebras. Then, they obtained $U_\vv(\mathfrak{sl}_n)$ (and its idempotented version) by applying a stabilization procedure as $d\to \infty$. Their construction gave a canonical basis for this quantum group and has inspired the work of several other authors. For example Grojnoski and Lusztig in \cite{GL} used analogous methods to describe in geometric terms the quantum Schur-Weyl duality due to Jimbo (\cite{J86}). 

There are multiple ways in which the work of BLM can be generalized. For example flag varieties for classical groups of type other than $A$ can be considered.

Let $d$ be a positive integer and $\mu=(\mu_1,\ldots, \mu_n)$ be a composition of $d$, i.e. $\mu_i$ is a nonnegative integer for all $i=1,\ldots, n$ and $\sum_i \mu_i=d$. Then notice that, for a field $\kk$, the space of all partial flags in $\kk^d$ with dimensions given by $\mu$, that is
$$\cF^\mu(\kk)=\{F=(0=F_0\subseteq F_1\subseteq \ldots\subseteq F_{n-1}\subseteq F_n=\kk^d~|~\dim(F_i/F_{i-1})=\mu_i\}$$
is isomorphic to the homogeneous space $\GL_d(\kk)/P^\mu(\kk)$ where $P^\mu(\kk)$ is the parabolic subgroup of all block upper triangular $d\times d$ matrices with blocks of sizes $(\mu_1,\ldots,\mu_n)$. It is then possible to replace $\GL_d(\kk)$ and $P^\mu(\kk)$ with  other classical groups and their parabolic subgroups. This has been done in recent work by Bao, Kujawa, Li and Wang \cite{BKLW} in type B/C and by Fan and Li in type D \cite{FL}.

Another direction of generalization, which we will focus on here, is passing to the `mirabolic' setting. This means that instead of considering pairs of partial flags, we take triples of two partial flags and a vector. The name comes from the mirabolic subgroup $P\subset \GL_d(\kk)$, which is the subgroup that fixes a nonzero vector in $\kk^d$. In general, for a $\GL_d$-variety $X$, the $P$-orbits on $X$ are in a 1-1 correspondence with $G$-orbits on $X\times (\kk^d\setminus\{0\})$. Mirabolic analogues of known constructions have been found to be interesting in several instances, for example mirabolic $\mathscr{D}$-modules arise when studying the spherical trigonometric Cherednik algebra (see \cite{FG}). Other examples are the enhanced nilpotent cone of \cite{AH} and the mirabolic RSK correspondence of \cite{T}.

\subsection{}The paper is organized as follows. In Section \ref{sec:conv-mir} we review the action of $\GL_d$ on triples of two partial flags and a vector and define a convolution product in this setting, in the same way as it was done for complete flags in \cite{R14}. This lets us define a mirabolic quantum Schur algebra $MU_\vv(n,d)$. Starting in Section \ref{sec:mirab-schur} we focus on the case $n=2$. We give some explicit formulae for computing convolution products in $MU_\vv(2,d)$ and identify a set of generators and some relations in this algebra. In Section \ref{sec:mir-quantum-sl2} we define $MU_\vv(2)$, the mirabolic version of the quantized enveloping algebra of $\mathfrak{sl}_2$, of which the $MU_\vv(2,d)$'s are finite dimensional quotients. We also find a PBW basis for this algebra. The category of finite dimensional $MU_\vv(2)$-representations is proved to be semisimple in Section \ref{sec:mir-reps} (using a mirabolic analogue of the Casimir element) and the irreducibles are classified. Finally, in Section \ref{sec:sw} we describe a mirabolic analogue of the quantum Schur-Weyl duality, which involves the mirabolic Hecke algebra $R_d$ of \cite{R14}. In the case $n=2$ we have a precise conjecture about the correspondence between irreducible representations of $MU_\vv(2)$ and of $R_d$.
\subsection{}Several interesting questions arise naturally from this work and will be the subject of future research.
\begin{itemize}
\item The quantum enveloping algebra $U_\vv(\mathfrak{sl}_n)$ for generic choices of the parameter $\vv$ behaves very similarly to the classical enveloping algebra $U(\mathfrak{sl}_n)$, but when $\vv$ is specialized to a root of unity things become more complicated. It is expected that $MU_\vv(2)$ will also display interesting behaviour when $\vv$ is a root of unity.
\item In this paper we only examine finite dimensional representations, but it should be possible to define Verma modules and a category $\mathscr{O}$  for $MU_\vv(2)$, in analogy with the case of $\uvsl$.
\item Of course we would like to generalize all the results to $n>2$. For $MU_\vv(2)$, as is explained in Section \ref{sec:mir-quantum-sl2}, we only need to add one generator $\ell$, which is an idempotent, to the generators of $\uvsl$. It is reasonable to expect that, just like in the case of the mirabolic Hecke algebra, even for $MU_n$ we should only need to add $\ell$ to the generators of $U_\vv(\mathfrak{sl}_n)$, and $\ell$ should commute with $e_i,f_i,k_i$, $i\geq 2$.
\end{itemize}
\subsection*{Notation}

We let $\N$ and $\N_+$ denote the set of nonnegative and positive integers respectively. We denote by $\F_q$ the finite field with $q$ elements. For a set $X$, we denote by $\# X$ its cardinality. If $d\in\N$, the notation $\lambda\vdash d$ means that $\lambda$ is a partition of $d$.

%

\subsection*{Acknowledgements}

The author would like to thank Victor Ginzburg for suggesting the line of research that led to this paper and for several helpful comments. He also thanks Jonas Hartwig for useful conversations. Finally, he is grateful to the University of California, Riverside for support.

%
\section{Convolution on mirabolic partial flag varieties.} \label{sec:conv-mir}
\subsection{$\GL_d$-orbits on partial flag varieties.}\label{subsec:orbits}
Let $\F_q$ be the finite field with $q$ elements. We fix positive integers $n,d$ 
and we consider the group $G_d:=\GL_d(\F_q)$ and the variety of all $n$-step partial flags in $\F_q^d$:
$$ \cF(n,d):=\{F=(0=F_0\subseteq F_1\subseteq\ldots\subseteq F_{n-1}\subseteq F_n=\F_q^d)\}.$$
The group $G_d$ acts naturally on $\F_q^d$ and this induces an action on $\cF(n,d)$. We consider the diagonal action of $G_d$ on $\cF(n,d)\times\cF(n,d)\times \F_q^d$, which has finitely many orbits.
These orbits have been parametrized in \cite{MWZ} 
in terms of ``decorated matrices'', as follows.
Let 
$$\Theta_{n,d}:=\{A=(a_{ij})\in M_n(\N)~|~ \txs\sum_{1\leq i,j,\leq n}a_{ij}=d\}$$
where $M_n(\N)$ denotes the set of $n\times n$ matrices with nonnegative integer entries. To a pair of flags $(F,F')\in\cF(n,d)^2$ we associate a matrix $A(F,F')=(a_{ij})\in\Theta_{n,d}$ with entries 
\begin{equation}a_{ij}=\dim\left(\frac{F_i\cap F'_j}{F_i\cap F'_{j-1}+F_{i-1}\cap F'_j}\right).\end{equation}
By \cite[1.1]{BLM}, this gives a bijection
$$ G_d\backslash \cF(n,d)\times \cF(n,d) \longleftrightarrow \Theta_{n,d}.$$
\begin{rem}\label{rem:basis}
A pair $(F,F')$ is in the orbit corresponding to a matrix $(a_{ij})$ if and only if there exists a basis $\{e_{ijk}~|~1\leq i,j\leq n~;~0<k\leq a_{ij}\}$ of $\F_q^d$ such that
$$ F_r=\langle e_{ijk} ~|~ 1\leq i\leq r, 0<k\leq a_{ij} \rangle~;\quad F'_s=\langle e_{ijk} ~|~ 1\leq j\leq s, 0<k\leq a_{ij} \rangle.$$
\end{rem}
\begin{defin}\label{def:dec-mat}We define a \emph{decorated matrix} to be a pair $(A,\Delta)$, where $A\in M_{n}(\N)$ and $\Delta=\{(i_1,j_1),\ldots,(i_k,j_k)\}$ is a (possibly empty) set that satisfies 
$$1\leq i_1<\ldots <i_k\leq n,\qquad 1\leq j_k<\ldots <j_1\leq n$$ 
and such that the entry $a_{ij}>0$ for all $(i,j)\in\Delta$. In particular, we consider a specific set of decorated matrices:
$$ \Xi_{n,d}:=\{(A,\Delta)~|~A\in\Theta_{n,d}\}.$$
\end{defin}
Then (see \cite[2.11]{MWZ}) we have a bijection
$$ G_d\backslash \cF(n,d)\times \cF(n,d)\times \F_q^d \longleftrightarrow \Xi_{n,d}.$$

\begin{rem}\label{rem:vect}We denote the orbit corresponding to the pair $(A,\Delta)$ by $\mathcal{O}_{A,\Delta}$. For a triple of two flags and a vector $(F,F',v)$, we have $(F,F',v)\in\mathcal{O}_{A,\Delta}$ if and only if there exists a basis as in Remark \ref{rem:basis} with the additional condition that $v=\sum_{(i,j)\in\Delta}e_{ij1}$.
\end{rem}
\begin{rem}Magyar, Weyman and Zelevinski actually consider the case of $G_d$-orbits on $\cF(n,d)\times\cF(n,d)\times\mathbb{P}(\F_q^d)$, which is equivalent to requiring that the vector in $\F_q^d$ be nonzero. Consequently their parametrization excludes the case where $\Delta=\emptyset$.
\end{rem}
We can concisely write down a pair $(A,\Delta)$, in a similar way to what is done in \cite{M}, by circling the entries of the matrix corresponding to $\Delta$.
\begin{exa}\label{ex:M1}$$A=\begin{pmatrix} 1 & 0 & 2 \\ 1 & 1 & 0 \\ 0 & 3 & 0 \end{pmatrix};\quad\Delta=\{(1,3),(2,1)\};\quad(A,\Delta)=\begin{pmatrix} 1 & 0 & \enc{2} \\ \enc{1} & 1 & 0 \\ 0 & 3 & 0 \end{pmatrix}. $$
\end{exa}

%
\subsection{Convolution product}
We consider $\mathcal{MU}_q(n,d):=\C(\cF(n,d)\times \cF(n,d)\times \F_q^d)^{G_d}$
, the space of $G_d$-invariant functions on the mirabolic partial flag variety. We define a convolution product as follows: if $\alpha,\beta\in\mathcal{MU}_q(n,d)$ then
\begin{equation}\label{eq:convolution}(\alpha*\beta)(F,F',v):=\sum_{H\in\cF(n,d),~u\in\F_q^d}\alpha(F,H,u)\beta(H,F',v-u).\end{equation}
Notice that the sum is finite because $\cF(n,d)$ and $\F_q^d$ are both finite sets, and \eqref{eq:convolution} defines an associative product on $\MU_q(n,d)$. This makes $\MU_q(n,d)$ into a finite dimensional associative algebra. If we denote by $T_{A,\Delta}$ the characteristic function of the orbit $\mathcal{O}_{A,\Delta}$, then the set $\{T_{A,\Delta}~|~(A,\Delta)\in\Xi_{n,d}\}$ is a basis of $\MU_q(n,d)$. 

For a matrix $A=(a_{ij})\in\Theta_{n,d}$, denote its row sums and column sums respectively by 
$$\ro(A)=(\txs\sum_{1\leq j\leq n}a_{1j},\ldots,\sum_{1\leq j\leq n}a_{nj});\quad\co(A)=(\sum_{1\leq i\leq n}a_{i1},\ldots,\sum_{1\leq i\leq n}a_{in}).$$
Then if the triple $(F,F',v)$ is in the orbit corresponding to $(A,\Delta)$, we have that
$$ \ro(A)=(\dim F_1,\dim(F_2/F_1),\ldots, \dim(\F_q^d/F_{n-1}));$$
$$ \co(A)=(\dim F'_1,\dim(F'_2/F'_1),\ldots, \dim(\F_q^d/F'_{n-1})).$$
It then follows immediately from \eqref{eq:convolution} that for all $(A,\Delta),(B,\Gamma)\in\Xi_{n,d}$, we have $T_{A,\Delta}*T_{B,\Gamma}=0$ if $\co(A)\neq\ro(B)$. Moreover, for a diagonal matrix $D\in\Theta_{n,d}$ we have
\begin{equation}\label{eq:diag-conv} T_{D,\emptyset}*T_{A,\Delta}=\delta_{\co(D),\ro(A)}T_{A,\Delta};\quad T_{A,\Delta}*T_{D,\emptyset}=\delta_{\co(A),\ro(D)}T_{A,\Delta},\end{equation}
where we have used Kronecker's $\delta$ notation.
From this observation, we see that $\MU_q(n,d)$ is a unital algebra and the unit element can be written in terms of the basis as
$$ 1=\sum_{D}T_{D,\emptyset},$$
where the sum runs over all diagonal matrices $D$ in $\Theta_{n,d}$.
\begin{defin}We call $\MU_q(n,d)$ the \emph{mirabolic quantum Schur algebra}.
\end{defin}
The name comes from the fact that this is the mirabolic analogue of the construction by Beilinson-Lusztig-MacPherson, as was mentioned in the introduction. Consider the space of invariant functions $\C(\cF(n,d)\times \cF(n,d))^{G_d}$  and define a convolution product, for $\alpha,\beta\in\C(\cF(n,d)\times \cF(n,d))^{G_d}$ by
\begin{equation}\label{eq:classic-conv}(\alpha*'\beta)(F,F'):=\sum_{H\in\cF(n,d)}\alpha(F,H)\beta(H,F').\end{equation}
Then the algebra we obtain is the \emph{quantum Schur algebra}, as is explained in \cite{BLM}.
 
 \begin{rem}\label{rem:inclusion}The inclusion 
 $$ \cF(n,d)\times \cF(n,d)\simeq\cF(n,d)\times \cF(n,d)\times \{0\}\rinto \cF(n,d)\times \cF(n,d)\times \F_q^d$$ 
 induces an embedding $i$ of the quantum Schur algebra into $\MU_q(n,d)$. 
It is given by identifying functions on pairs of flags with functions supported on the subspace of triples where the vector is $0$, that is, for all $\alpha\in \C(\cF(n,d)\times \cF(n,d))^{G_d}$, we get
$$ i(\alpha)(F,F',v)=\left\{\begin{array}{cl} \alpha(F,F') & \text{ if }v=0; \\ 0 & \text{ if }v\neq 0.\end{array}\right.$$
From the definition of the products in \eqref{eq:convolution} and \eqref{eq:classic-conv}, it is clear that $i(\alpha*'\beta)=i(\alpha)*i(\beta)$ so this is indeed an embedding of algebras.
\end{rem}
\begin{rem}\label{rem:anti-inv} The involution on $\cF(n,d)\times \cF(n,d)\times \F_q^d$ defined by $(F,F',v)\mapsto(F',F,v)$ induces an algebra anti-automorphism $^\star:\MU_q(n,d)\to\MU_q(n,d)$. In the natural basis for $\MU_q(n,d)$, this can be written as $(T_{A,\Delta})^\star= T_{^tA,^t\Delta}$, where $^tA$ denotes the transpose matrix and $^t\Delta$ corresponds to keeping track of where the marked positions on the matrix have moved to after transposition. More precisely, if $\Delta=\{(i_1,j_1),\ldots,(i_k,j_k)\},$ then $^t\Delta=\{(j_k,i_k),\ldots,(j_1,i_1)\}$.
\end{rem}
\begin{defin}\label{def:extend-scalars}The structure constants for the multiplication in $\MU_q(n,d)$ are polynomials in $\Z[q]$, hence we can consider $\MU_q(n,d)$ to be the specialization at $\q\mapsto q$ of a $\C[\q,\q^{-1}]$ algebra $\MU_\q(n,d)$. We then extend scalars and define
$$MU_\vv(n,d):=\C(\vv)\otimes_{\C[\q,\q^{-1}]} \MU_\q(n,d),$$ 
where the map $\C[\q,\q^{-1}]\to\C(\vv)$ is given by $\q\mapsto \vv^2$. 

We call $MU_\vv(n,d)$ the \emph{generic mirabolic quantum Schur algebra}.
\end{defin}
By abuse of notation, we will denote the basis elements of $MU_\vv(n,d)$ as $T_{A,\Delta}$ in the same way as the ones in $\MU_q(n,d)$, and analogously for the anti-involution of Remark \ref{rem:anti-inv}.
\section{Algebra structure of $\MU_q(2,d)$ and $MU_\vv(2,d)$.}\label{sec:mirab-schur}
We now focus on the case $n=2$. In this case, given a $2\times 2$ matrix $A\in\Theta_{2,d}$, we have at most six possibilities for $\Delta$, such that $(A,\Delta)\in\Xi_{2,d}$, namely $\Delta=\emptyset$, $\{(1,1)\}$, $\{(1,2)\}$, $\{(2,1)\}$, $\{(1,2),(2,1)\}$. $\{(2,2)\}$. Visually, these are the possibilities for $(A,\Delta)$ (assuming that the appropriate entries are nonzero):
\begin{equation}\label{eq:possib-matr}\begin{pmatrix} a_{11} & a_{12} \\ a_{21} & a_{22} \end{pmatrix};~\begin{pmatrix} \enc{a_{11}} & a_{12} \\ a_{21} & a_{22} \end{pmatrix};~\begin{pmatrix} a_{11} & \enc{a_{12}} \\ a_{21} & a_{22} \end{pmatrix};~\begin{pmatrix} a_{11} & a_{12} \\ \enc{a_{21}} & a_{22} \end{pmatrix};~\begin{pmatrix} a_{11} & \enc{a_{12}} \\ \enc{a_{21}} & a_{22} \end{pmatrix};~\begin{pmatrix} a_{11} & a_{12} \\ a_{21} & \enc{a_{22}} \end{pmatrix}.\end{equation}
Geometrically, we have that if $(F,F',v)\in\mathcal{O}_{A,\Delta}$ for the various cases in \eqref{eq:possib-matr}, then the vector $v$ satisfies the following conditions, respectively:
$$v=0;\qquad 0\neq v\in F_1\cap F'_1;\qquad  v\in F_1\setminus F_1\cap F'_1;$$
$$v\in F'_1\setminus F_1\cap F'_1;\qquad v\in (F_1+F'_1)\setminus (F_1\cup F'_1);\qquad v\in\F_q^d\setminus (F_1+F'_1).$$
\subsection{Multiplication Formulas}
We denote by $E_{i,j}\in M_n(\N)$ the elementary matrix with $1$ in the $(i,j)$-entry and zeros everywhere else. We are going to do some computations in $\MU_q(2,d)$, but then these clearly imply the analogous statements for $MU_\vv(2,d)$ (after replacing $q$ with $\vv^2$). Remember that we denote by $T_{A,\Delta}$ the characteristic function of the orbit $\mathcal{O}_{A,\Delta}$, where $A=(a_{ij})$.
\begin{prop}\label{prop:comp-e-action}Suppose that $B,A\in\Theta_{2,d}$ such that $\ro(A)=\co(B)$ and $B-E_{1,2}$ is a diagonal matrix, then we have

\begin{itemize}
\item[(a)] $T_{B,\emptyset}*T_{A,\emptyset}=q^{a_{12}}\df{q^{a_{11}+1}-1}{q-1}T_{A+E_{1,1}-E_{2,1},\emptyset}+\df{q^{a_{12}+1}-1}{q-1}T_{A+E_{1,2}-E_{2,2},\emptyset};$
\item[]
\item[(b)] $T_{B,\emptyset}*T_{A,\{(1,1)\}}=q^{a_{12}}\df{q^{a_{11}}-1}{q-1}T_{A+E_{1,1}-E_{2,1},\{(1,1)\}}+\df{q^{a_{12}+1}-1}{q-1}T_{A+E_{1,2}-E_{2,2},\{(1,1)\}};$
\item[]
\item[(c)] $T_{B,\emptyset}*T_{A,\{(1,2)\}}=q^{a_{12}}\df{q^{a_{11}}-1}{q-1}T_{A+E_{1,1}-E_{2,1},\{(1,2)\}}+\df{q^{a_{12}}-1}{q-1}T_{A+E_{1,2}-E_{2,2},\{(1,2)\}};$
\item[]
\item[(d)] $T_{B,\emptyset}*T_{A,\{(2,1)\}}=q^{a_{12}}q^{a_{11}}T_{A+E_{1,1}-E_{2,1},\{(1,1)\}}+q^{a_{12}}\df{q^{a_{11}+1}-1}{q-1}T_{A+E_{1,1}-E_{2,1},\{(2,1)\}}\\ \qquad\qquad+\df{q^{a_{12}+1}-1}{q-1}T_{A+E_{1,2}-E_{2,2},\{(2,1)\}};$
\item[]
\item[(e)] $T_{B,\emptyset}*T_{A,\{(1,2),(2,1)\}}=q^{a_{12}}q^{a_{11}}T_{A+E_{1,1}-E_{2,1},\{(1,2)\}}+q^{a_{12}}\df{q^{a_{11}+1}-1}{q-1}T_{A+E_{1,1}-E_{2,1},\{(1,2),(2,1)\}}\\ 
\qquad\qquad+\df{q^{a_{12}}-1}{q-1}T_{A+E_{1,2}-E_{2,2},\{(1,2),(2,1)\}}$
\item[]
\item[(f)] $T_{B,\emptyset}*T_{A,\{(2,2)\}}=q^{a_{12}}T_{A+E_{1,2}-E_{2,2},\{(1,2)\}}+q^{a_{12}}T_{A+E_{1,2}-E_{2,2},\{(1,2),(2,1)\}}\\ \qquad\qquad +q^{a_{12}}\df{q^{a_{11}+1}-1}{q-1}T_{A+E_{1,1}-E_{2,1},\{(2,2)\}}+\df{q^{a_{12}+1}-1}{q-1}T_{A+E_{1,2}-E_{2,2},\{(2,2)\}};$
\item[]
\end{itemize}
Here we interpret $T_{A,\Delta}$ for any $(A,\Delta)\not\in\Xi_{2,d}$ as zero.
\end{prop}
\begin{proof}In what follows, by the notation $W\subcod{r} V$ we mean that $W$ is a subspace of $V$ with codimension $r$.
\begin{itemize}
\item[(a)] Given the inclusion of Remark \ref{rem:inclusion}, this is just a special case of \cite[Lemma~3.2(a)]{BLM}.

\item[(b)] Let us fix a triple $(F,F',v)$. What we need to do is count the set 
\begin{equation}\label{eq:count1}\{H\in\cF(2,d)~|~(F,H,0)\in\mathcal{O}_{B,\emptyset}~\text{ and }(H,F',v)\in\mathcal{O}_{A,\{(1,1)\}}\}.\end{equation}
Notice that since $F, H, F'$ are two step flags, they are completely determined by $F_1,H_1,F'_1$ respectively. Now, the condition on $F$ and $H$ means that $F_1\stackrel{1}{\supset}H_1$, while the condition on $H,F'$ and $v$ means that $0\neq v\in H_1\cap F'_1$. This clearly implies that $0\neq v\in H_1\cap F'_1\subset F_1\cap F'_1$ so the terms appearing will all involve $\Delta=\{(1,1)\}$. There are two possibilities: $F_1\cap F'_1=H_1\cap F'_1$ or $F_1\cap F'_1\supcod{1} H_1\cap F'_1$. In the first case the relative position of $F$ and $F'$ has to be $A+E_{1,2}-E_{2,2}$ and $v\in F_1\cap F'_1$ if an only if $v\in H_1\cap F'_1$, so the number of $H$'s that satisfy \eqref{eq:count1} is the same as in part (a). 
In the second case, the relative position of $F$ and $F'$ is $A+E_{1,1}-E_{2,1}$ and we need to count the $H$'s such that $v \in H_1\cap F'_1$, which is equal to $q^{a_{12}}\frac{q^{a_{11}}-1}{q-1}$.

\item[(c)] Here we are counting $H$'s such that
\begin{equation}\label{eq:cond-c} v\in H_1\setminus H_1\cap F'_1.\end{equation} 
Clearly for that to be true we need $v\in F_1\setminus F_1\cap F'_1$, so let us fix such a $v$. If $H_1\cap F'_1=F_1\cap F'_1$, then \eqref{eq:cond-c} is true when $(F_1\cap F'_1)\oplus \F_qv\subset H_1\subcod{1} F_1$, so we get $\frac{q^{a_{12}}-1}{q-1}$. If $F_1\cap F'_1\supcod{1}H_1\cap F'_1$, then we count
\begin{align*}&\#\{H~|~H_1\subcod{1}F_1\text{ and }v\in H_1\}-\#\{H~|~H_1\subcod{1}F_1,~v\in H_1\text{ and }H_1\cap F'_1=F_1\cap F'_1\} \\
&= \df{q^{a_{11}+a_{12}}-1}{q-1}-\df{q^{a_{12}}-1}{q-1} = q^{a_{12}}\df{q^{a_{11}}-1}{q-1}.
\end{align*}

\item[(d)] In this case, we need 
\begin{equation}\label{eq:cond1}v\in F'_1\setminus H_1\cap F'_1.\end{equation} 
When $F_1\cap F'_1=H_1\cap F'_1$ then \eqref{eq:cond1} is equivalent to $v\in F'_1\setminus F_1\cap F'_1$ for all possible choices of $H$. When $F_1\cap F'_1\supcod{1}H_1\cap F'_1$ there are two more possibilities. If $v\in F'\setminus F_1\cap F'_1$ then \eqref{eq:cond1} is satisfied for all possible $H$'s. If $v\in F'_1\cap F_1$ we need to count the $H$ such that $v\not\in H_1\cap F'_1$ which is
\begin{align*}& \#\{ H~|~F_1\cap F'_1\supcod{1} H_1\cap F'_1\}-\#\{H~|~F_1\cap F'_1\supcod{1} H_1\cap F'_1\text{ and }v\in H_1\cap F'_1\} \\
&= q^{a_{12}}\df{q^{a_{11}+1}-1}{q-1}-q^{a_{12}}\df{q^{a_{11}}-1}{q-1}= q^{a_{12}}q^{a_{11}}.
\end{align*}

\item[(e)] Here the condition is
\begin{equation}\label{eq:cond-e}v\in (H_1+F'_1)\setminus (H_1\cup F'_1).\end{equation}
This can happen in two cases: (1) when $v\in (F_1+F'_1)\setminus (F_1\cup F'_1)$ and (2) when $v\in F_1\setminus F_1\cap F'_1$.
In case (1), clearly $v\not\in H_1$, so all we need to check for \eqref{eq:cond-e} is whether $v\in H_1+F'_1$. If $F_1\cap F'_1\supcod{1}H_1\cap F'_1$ then for all possible $H$'s we have indeed $v\in H_1+F'_1$. If $F_1\cap F'_1=H_1\cap F'_1$, then the extra condition given by $v$ cuts down one dimension of possible $H$'s, so we get $\frac{q^{a_{12}}-1}{q-1}$. For case (2), if $F_1\cap F'_1=H_1\cap F'_1$ then $v\in H_1+F'_1$ if and only if $v\in H$, so no choice of $H$ can satisfy \eqref{eq:cond-e}. Finally, when $F_1\cap F'_1\supcod{1}H_1\cap F'_1$ then $H_1+F'_1\subset F_1$ so we need to count the $H$'s such that $v\not\in H_1$, this is the opposite computation of what we did in part (c), so we get
\begin{align*}&\#\{H~|~H_1\cap F'_1\subcod{1}F_1\cap F'_1\}-\#\{H~|~H_1\cap F'_1\subcod{1}F_1\cap F'_1\text{ and }v\in H_1\} \\
&=q^{a_{12}}\df{q^{a_{11}+1}-1}{q-1}- q^{a_{12}}\df{q^{a_{11}}-1}{q-1}=q^{a_{12}}q^{a_{11}}.
\end{align*}

\item[(f)] Now we want
\begin{equation}\label{eq:cond-f}v\not\in H_1+F'_1.\end{equation}
There are three possiblities here: $v\not\in F_1+F'_1$, $v\in (F_1+F'_1)\setminus (F_1\cup F'_1)$, and $v\in F_1\setminus F_1\cap F'_1$. If $v\not\in  F_1+F'_1$ then \eqref{eq:cond-f} is satisfied for all choices of $H$. If $v\in (F_1+F'_1)\setminus (F_1\cup F'_1)$ then necessarily $H_1\cap F'_1=F_1\cap F'_1$ (in the other case $H_1+F'_1=F_1+F'_1$ so \eqref{eq:cond-f} cannot be true). In this case, then we are counting (using part (e))
\begin{align*}&\#\{H~|~H_1\cap F'_1=F_1\cap F'_1\}-\#\{H~|~H_1\cap F'_1=F_1\cap F'_1\text{ and }v\in H_1+F'_1\} \\
&=\df{q^{a_{12}+1}-1}{q-1}- \df{q^{a_{12}}-1}{q-1}=q^{a_{12}}.
\end{align*}
Finally, if $v\in F_1\setminus F_1\cap F'_1$ then again $H_1\cap F'_1\subcod{1}F_1\cap F'_1$ is not an option because that implies that $F_1\subset H_1+F'_1$, which makes \eqref{eq:cond-f} impossible. So, when $H_1\cap F'_1=F_1\cap F'_1$ we get that \eqref{eq:cond-f} is true if and only if $v\not\in H_1$, hence using part (c) we count
\begin{align*}&\#\{H~|~H_1\cap F'_1=F_1\cap F'_1\}-\#\{H~|~H_1\cap F'_1=F_1\cap F'_1\text{ and }v\in H_1\} \\
&=\df{q^{a_{12}+1}-1}{q-1}- \df{q^{a_{12}}-1}{q-1}=q^{a_{12}}.
\end{align*}
\end{itemize}
\end{proof}
\begin{prop}\label{prop:comp-f-action}Suppose that $C,A\in\Theta_{2,d}$ such that $\ro(A)=\co(C)$ and $C-E_{2,1}$ is a diagonal matrix, then we have

\begin{itemize}
\item[(a)] $T_{C,\emptyset}*T_{A,\emptyset}=\df{q^{a_{21}+1}-1}{q-1}T_{A+E_{2,1}-E_{1,1},\emptyset}+q^{a_{21}}\df{q^{a_{22}+1}-1}{q-1}T_{A+E_{2,2}-E_{1,2},\emptyset};$
\item[]
\item[(b)] $T_{C,\emptyset}*T_{A,\{(1,1)\}}=\df{q^{a_{21}+1}-1}{q-1}T_{A+E_{2,1}-E_{1,1},\{(1,1)\}}+q^{a_{21}}\df{q^{a_{22}+1}-1}{q-1}T_{A+E_{2,2}-E_{1,2},\{(1,1)\}}\\
\qquad +T_{A+E_{2,1}-E_{1,1},\{(2,1)\}};$
\item[]
\item[(c)] $T_{C,\emptyset}*T_{A,\{(1,2)\}}=\df{q^{a_{21}+1}-1}{q-1}T_{A+E_{2,1}-E_{1,1},\{(1,2)\}}+q^{a_{21}}\df{q^{a_{22}+1}-1}{q-1}T_{A+E_{2,2}-E_{1,2},\{(1,2)\}}\\
\qquad +T_{A+E_{2,1}-E_{1,1},\{(1,2),(2,1)\}}+T_{A+E_{2,2}-E_{1,2},\{(2,2)\}};$
\item[]
\item[(d)] $T_{C,\emptyset}*T_{A,\{(2,1)\}}=q\df{q^{a_{21}}-1}{q-1}T_{A+E_{2,1}-E_{1,1},\{(2,1)\}}+q^{a_{21}}\df{q^{a_{22}+1}-1}{q-1}T_{A+E_{2,2}-E_{1,2},\{(2,1)\}}$
\item[]
\item[(e)] $T_{C,\emptyset}*T_{A,\{(1,2),(2,1)\}}=q\df{q^{a_{21}}-1}{q-1}T_{A+E_{2,1}-E_{1,1},\{(1,2),(2,1)\}}+q^{a_{21}}\df{q^{a_{22}+1}-1}{q-1}T_{A+E_{2,2}-E_{1,2},\{(1,2),(2,1)\}}\\
\qquad+q^{a_{21}}T_{A+E_{2,2}-E_{1,2},\{(2,2)\}};$
\item[]
\item[(f)] $T_{C,\emptyset}*T_{A,\{(2,2)\}}=\df{q^{a_{21}+1}-1}{q-1}T_{A+E_{2,1}-E_{1,1},(2,2)}+\left(q^{a_{21}+1}\df{q^{a_{22}}-1}{q-1}-1\right)T_{A+E_{2,2}-E_{1,2},(2,2)}.$
\item[]
\end{itemize}
Here we interpret $T_{A,\Delta}$ for any $(A,\Delta)\not\in\Xi_{2,d}$ as zero.

\end{prop}
\begin{proof}The arguments here are entirely analogous to the ones in the proof of Proposition \ref{prop:comp-e-action} and will be omitted (see also \cite[Lemma~3.2(b)]{BLM}).
\end{proof}
\begin{prop}\label{prop:comp-x-action}Suppose that $D,A\in\Theta_{2,d}$ such that $\ro(A)=\co(D)$ and $D$ is a diagonal matrix, then we have

\begin{itemize}
\item[(a)] $T_{D,\{(1,1)\}}*T_{A,\emptyset}=T_{A,\{(1,1)\}}+T_{A,\{(1,2)\}};$
\item[]
\item[(b)] $T_{D,\{(1,1)\}}*T_{A,\{(1,1)\}}=(q^{a_{11}}-1)T_{A,\emptyset}+(q^{a_{11}}-2)T_{A,\{(1,1)\}}+(q^{a_{11}}-1)T_{A,\{(1,2)\}};$
\item[]
\item[(c)] $T_{D,\{(1,1)\}}*T_{A,\{(1,2)\}}=q^{a_{11}}(q^{a_{12}}-1)T_{A,\emptyset}+q^{a_{11}}(q^{a_{12}}-1)T_{A,\{(1,1)\}}+\\ +(q^{a_{11}}(q^{a_{12}}-1)-1)T_{A,\{(1,2)\}};$
\item[]
\item[(d)] $T_{D,\{(1,1)\}}*T_{A,\{(2,1)\}}=(q^{a_{11}}-1)T_{A,\{(2,1)\}}+q^{a_{11}}T_{A,\{(1,2),(2,1)\}};$
\item[]
\item[(e)] $T_{D,\{(1,1)\}}*T_{A,\{(1,2),(2,1)\}}=q^{a_{11}}(q^{a_{12}}-1)T_{A,\{(2,1)\}}+(q^{a_{11}}(q^{a_{12}}-1)-1)T_{A,\{(1,2),(2,1)\}};$
\item[]
\item[(f)] $T_{D,\{(1,1)\}}*T_{A,\{(2,2)\}}=(q^{a_{11}+a_{12}}-1)T_{A,\{(2,2)\}};$
\item[]
\item[(g)]$T_{D,\{(2,2)\}}*T_{A,\emptyset}=T_{A,\{(2,1)\}}+T_{A,\{(1,2),(2,1)\}}+T_{A,\{(2,2)\}};$
\item[]
\item[(h)]$T_{D,\{(2,2)\}}*T_{A,\{(1,1)\}}=(q^{a_{11}}-1)\left(T_{A,\{(2,1)\}}+T_{A,\{(1,2),(2,1)\}}+T_{A,\{(2,2)\}}\right);$
\item[]
\item[(i)]$T_{D,\{(2,2)\}}*T_{A,\{(1,2)\}}=q^{a_{11}}(q^{a_{12}}-1)\left(T_{A,\{(2,1)\}}+T_{A,\{(1,2),(2,1)\}}+T_{A,\{(2,2)\}}\right);$
\item[]
\item[(j)]$T_{D,\{(2,2)\}}*T_{A,\{(2,1)\}}=q^{a_{11}}(q^{a_{21}}-1)\left(T_{A,\emptyset}+T_{A,\{(1,1)\}}+T_{A,\{(1,2)\}}+T_{A,\{(2,2)\}}\right)\\
\qquad +q^{a_{11}}(q^{a_{21}}-2)\left(T_{A,\{(2,1)\}}+T_{A,\{(1,2),(2,1)\}}\right); $
\item[]
\item[(k)]$T_{D,\{(2,2)\}}*T_{A,\{(1,2),(2,1)\}}=q^{a_{11}}(q^{a_{12}}-1)(q^{a_{21}}-1)\left(T_{A,\emptyset}+T_{A,\{(1,1)\}}+T_{A,\{(1,2)\}}+T_{A,\{(2,2)\}}\right)\\
\qquad +q^{a_{11}}(q^{a_{12}}-1)(q^{a_{21}}-2)\left(T_{A,\{(2,1)\}}+T_{A,\{(1,2),(2,1)\}}\right); $
\item[]
\item[(l)]$T_{D,\{(2,2)\}}*T_{A,\{(2,2)\}}=q^{a_{11}+a_{12}+a_{21}}(q^{a_{22}}-2)T_{A,\{(2,2)\}}\\
\qquad+ q^{a_{11}+a_{12}+a_{21}}(q^{a_{22}}-1)\left(T_{A,\emptyset}+T_{A,\{(1,1)\}}+T_{A,\{(1,2)\}}+T_{A,\{(2,1)\}}+T_{A,\{(1,2),(2,1)\}}\right) .$
\item[]
\end{itemize}

Here we interpret $T_{A,\Delta}$ for any $(A,\Delta)\not\in\Xi_{2,d}$ as zero.
\end{prop}
\begin{proof}
Again the arguments are very similar to the proof of Proposition \ref{prop:comp-e-action} and will be omitted.
\end{proof}
\subsection{Generators and some relations}
In this section we will assume that we are working in $MU_\vv(2,d)$ and for simplicity we will denote the product $*$ by juxtaposition. For each $m\in\N$ we define the quantum symmetric integer and quantum factorial (by convention $[0]!=1$)
$$[m]:=\df{\vv^m-\vv^{-m}}{\vv-\vv^{-1}}=\vv^{m-1}+\vv^{m-3}+\ldots+\vv^{-m+3}+ \vv^{-m+1};\qquad\qquad [m] !=\prod_{k=1}^m [k].$$ 
As before, let $E_{i,j}\in M_2(\N)$ be the elementary matrix and, for $r,s\in\N$, let $D(r,s):=\begin{pmatrix} r & 0 \\ 0 & s \end{pmatrix}$. 
We define the following elements of $MU_\vv(2,d)$:
\begin{align*}
e_r& :=T_{D(r,d-r-1)+E_{1,2},\emptyset} &\text{ for all }0\leq r\leq d-1, \\
f_r& :=T_{D(r,d-r-1)+E_{2,1},\emptyset} &\text{ for all }0\leq r\leq d-1, \\
1_r&:=T_{D(r,d-r),\emptyset}&\text{ for all }0\leq r\leq d, \\
x_r&:=T_{D(r,d-r),\{(1,1)\}}&\text{ for all }1\leq r\leq d. 
\end{align*}
\begin{lem}\label{lem:2-2}For all $0\leq r\leq d-1$ we have
\begin{align*} T_{D(r,d-r),\{(2,2)\}}&=f_rx_{r+1}e_r-\vv^{2-2r}(x_rf_re_rx_r+x_rf_re_r+f_re_rx_r+f_re_r)+f_re_r+\\
&+(\vv^{d-r-1}-\vv^{d-r+1})[d-r](1_r+x_r).\end{align*}
\end{lem}
\begin{proof}
By  repeated applications of Propositions \ref{prop:comp-e-action}, \ref{prop:comp-f-action}, and \ref{prop:comp-x-action} it can be readily checked that for all $0\leq r\leq  d-1$ we have
\begin{align*}
f_rx_{r+1}e_r&=\vv^{d-r-1}[d-r] x_r+T_{D(r,d-r),\{(2,2)\}}+T_{D(r-1,d-r-1)+E_{1,2}+E_{2,1},\emptyset}+\\
 &+T_{D(r-1,d-r-1)+E_{1,2}+E_{2,1},\{(1,1)\}}+T_{D(r-1,d-r-1)+E_{1,2}+E_{2,1},\{(1,2)\}}+\\
 &+T_{D(r-1,d-r-1)+E_{1,2}+E_{2,1},\{(1,2),(2,1)\}};\\
x_rf_re_r&=\vv^{d-r-1}[d-r] x_r+T_{D(r-1,d-r-1)+E_{1,2}+E_{2,1},\{(1,1)\}}+T_{D(r-1,d-r-1)+E_{1,2}+E_{2,1},\{(1,2)\}};\\
f_re_rx_r&=\vv^{d-r-1}[d-r] x_r+T_{D(r-1,d-r-1)+E_{1,2}+E_{2,1},\{(1,1)\}}+T_{D(r-1,d-r-1)+E_{1,2}+E_{2,1},\{(2,1)\}};\\
x_rf_re_rx_r&=(\vv^{d+r-1}-\vv^{d-r-1})[d-r] 1_r+(\vv^{d+r-1}-2\vv^{d-r-1})[d-r] x_r+\\ 
 &+(\vv^{2r-2}-1)T_{D(r-1,d-r-1)+E_{1,2}+E_{2,1},\emptyset}+(\vv^{2r-2}-2)T_{D(r-1,d-r-1)+E_{1,2}+E_{2,1},\{(1,1)\}}+\\
 &+(\vv^{2r-2}-1)T_{D(r-1,d-r-1)+E_{1,2}+E_{2,1},\{(1,2)\}}+(\vv^{2r-2}-1)T_{D(r-1,d-r-1)+E_{1,2}+E_{2,1},\{(2,1)\}}+\\
 &+\vv^{2r-2}T_{D(r-1,d-r-1)+E_{1,2}+E_{2,1},\{(1,2),(2,1)\}}.
\end{align*}
The result then follows.
\end{proof} 

\begin{theo}\label{theo:generators}The following elements:
$$ e:=\sum_{r=0}^{d-1}\vv^{-r}e_r;\qquad f:=\sum_{r=0}^{d-1}\vv^{1+r-d}f_r;$$
$$ k:=\sum_{r=0}^d\vv^{2r-d}1_r;\qquad k^{-1}:=\sum_{r=0}^d\vv^{d-2r}1_r;$$
$$\ell:=1_0+\sum_{r=1}^d \vv^{-2r}(1_r+x_r); $$
generate the $\C(\vv)$-algebra $MU_\vv(2,d)$.

We call $e,f,k,k^{-1},\ell$ the \emph{Chevalley generators} of $MU_\vv(2,d)$.
\end{theo}
\begin{proof}Let $M'$ be the $\C(\vv)$-algebra generated by $e,f,k,k^{-1},\ell$. We will prove that $T_{A,\Delta}\in M'$ for all $(A,\Delta)\in \Xi_{2,d}$. Note that $1_r1_s=\delta_{r,s}1_r$, hence for all $m=0,\ldots, d$ we have that $k^m=\sum_{r=0}^d(\vv^{2r-d})^m1_r$. Because of the Vandermonde determinant (since $\vv^{2r-d}\neq \vv^{2s-d}$ when $r\neq s$), we have that the set $\{k^m~|~m=0,\ldots, d\}$ is linearly independent, hence it forms a basis of $\Span\{ 1_r~|~r=0,\ldots d\}$. It follows that $1_r\in M'$ for all $r=0,\ldots d$. It then follows that $e_r,f_r\in M'$ for $r=0,\ldots,d-1$ because $M'\ni e1_r=\vv^{-r}e_r$ and $M'\ni 1_rf=\vv^{1+r-d}f_r$. We also have $x_r=\vv^{2r}\ell 1_r-1_r\in M'$ for $r=1,\ldots, d$. By Lemma \ref{lem:2-2} we then have $ T_{D(r,d-r),\{(2,2)\}}\in M'$ for $r=0,\ldots d-1$. Therefore we have proved that $T_{D,\Delta}\in M'$ for all $(D,\Delta)\in \Xi_{2,d}$ where $D$ is a diagonal matrix since, for such a $D$, the only options for $\Delta$ are $\emptyset$, $\{(1,1)\}$ or $\{(2,2)\}$.

Notice that the fact that $T_{A,\emptyset}\in M'$ for all $A$ follows from the inclusion of Remark \ref{rem:inclusion} and the fact proved in \cite{BLM} that $e$, $f$, $k$, $k^{-1}$ generate the quantum Schur algebra, but we will see this directly.

Suppose that $A$ is upper triangular,
$$A=\begin{pmatrix} r & m \\ 0 & d-r-m \end{pmatrix}\quad\text{ with $m\geq 1$, then }\quad T_{A,\emptyset}=\frac{\vv^{-{m \choose 2}}}{[m] !}e_{r+m-1}\cdots e_{r+1}e_r\in M'$$
by Proposition \ref{prop:comp-e-action}(a).
Then, using Remark \ref{rem:anti-inv} and Prop. \ref{prop:comp-x-action}(a), we get that, if $r=1,\ldots,d-m$,
$$ T_{A,\emptyset}x_r=(x_r^\star (T_{A,\emptyset})^\star)^\star=(x_rT_{^t {A},\emptyset})^\star=(T_{^tA,\{(1,1)\}})^\star=T_{A,\{(1,1)\}}\in M';$$
and also
$$ x_{r+m}T_{A,\emptyset}-T_{A,\emptyset}x_r=T_{A,\{(1,1)\}}+T_{A,\{(1,2)\}}-T_{A\{(1,1)\}}=T_{A,\{(1,2)\}}\in M'.$$
In the case where $r=0$, $x_r$ and $T_{A,\{(1,1)\}}$ do not exist, but we still get  $x_{r+m}T_{A,\emptyset}=T_{A,\{(1,2)\}}$.
Finally, when $d-r-m\geq 1$, Prop. \ref{prop:comp-x-action}(g) implies that
$$ T_{D(r+m,d-r-m),\{(2,2)\}}T_{A,\emptyset}=T_{A,\{(2,2)\}}\in M'$$
which proves that $T_{A,\Delta}\in M'$ for all $(A,\Delta)\in\Xi_{2,d}$ where $A$ is upper triangular.

Now take any $A=\begin{pmatrix} a_{11} & a_{12} \\ a_{21} & a_{22} \end{pmatrix}\in \Theta_{2,d}$, we will argue by induction on $a_{21}$ that $T_{A,\Delta}\in M'$ for all choices of $\Delta$. The base case is when $a_{21}=0$, that is $A$ is upper triangular, which we have already discussed.
Suppose $a_{21}\geq 1$, and let 
$$A'=\begin{pmatrix} a_{11} & a_{12} \\ a_{21}-1 & a_{22}+1 \end{pmatrix};\quad A''=\begin{pmatrix} a_{11}+1 & a_{12}-1 \\ a_{21}-1 & a_{22}+1 \end{pmatrix}.$$
(If $a_{12}=0$, we will take $T_{A'',\Delta}=0$ in what follows.)

Observe that, by Remark \ref{rem:anti-inv} and Prop. \ref{prop:comp-e-action}(a)
$$ T_{A',\emptyset}f_{a_{11}+a_{21}-1}=(e_{a_{11}+a_{21}-1}T_{^t A',\emptyset})^\star=\vv^{2a_{21}+a_{11}-2}[a_{11}+1] T_{A'',\emptyset}+\vv^{a_{21}-1}[a_{21}] T_{A,\emptyset}.$$
By inductive hypothesis, $ T_{A',\emptyset}, T_{A'',\emptyset}\in M'$ therefore $ T_{A,\emptyset}\in M'$.
Then, using Prop. \ref{prop:comp-e-action}(b) we have
\begin{align*} T_{A',\{(1,1)\}}f_{a_{11}+a_{21}-1}&=(e_{a_{11}+a_{21}-1}T_{^t A',\{(1,1)\}})^\star\\
&=\vv^{2a_{21}+a_{11}-3}[a_{11}] T_{A'',\{(1,1)\}}+\vv^{a_{21}-1}[a_{21}] T_{A,\{(1,1)\}}.\end{align*}
Again, by the inductive hypothesis, $ T_{A',\{(1,1)\}}, T_{A'',\{(1,1)\}}\in M'$ hence $ T_{A,\{(1,1)\}}\in M'$.
By Proposition \ref{prop:comp-e-action}(d), we have
\begin{multline*} T_{A',\{(1,2)\}}f_{a_{11}+a_{21}-1}=(e_{a_{11}+a_{21}-1}T_{^t A',\{(2,1)\}})^\star\\
=\vv^{2a_{21}+2a_{11}-2} T_{A'',\{(1,1)\}}+\vv^{2a_{21}+a_{11}-2}[a_{11}+1] T_{A'',\{(1,2)\}}+\vv^{a_{21}-1}[a_{21}] T_{A,\{(1,2)\}}.\end{multline*}
By the previous computation, $T_{A'',\{(1,1)\}}\in M'$ and by induction $T_{A',\{(1,2)\}},T_{A'',\{(1,2)\}}\in M'$, hence $T_{A,\{(1,2)\}}\in M'$.

Now let 

$$ A'''=\begin{pmatrix} a_{11}+1 & a_{12} \\ a_{21}-1 & a_{22} \end{pmatrix}.$$
By Prop. \ref{prop:comp-f-action}(b) we get
$$f_{a_{11}+a_{12}}T_{A''',\{(1,1)\}}=\vv^{a_{21}-1}[a_{21}] T_{A,\{(1,1)\}}+T_{A,\{(2,1)\}}+\vv^{2a_{21}+a_{22}-2}[a_{22}+1] T_{A'',\{(1,1)\}}.$$
By the previous computation $T_{A,\{(1,1)\}},T_{A'',\{(1,1)\}},T_{A''',\{(1,1)\}}\in M'$, so $T_{A,\{(2,1)\}}\in M'$.


Now suppose $a_{12}\geq 1$, and $a_{21}=1$, then by Prop. \ref{prop:comp-f-action}(c) 
$$ f_{a_{11}+a_{12}}T_{A''',\{(1,2)\}}=T_{A,\{(1,2)\}}+\vv^{2a_{21}+a_{22}-2}[a_{22}+1] T_{A'',\{(1,2)\}}+T_{A,\{(1,2),(2,1)\}}+T_{A'',\{(2,2)\}}.$$
Since we have seen that $T_{A,\{(1,2)\}}, T_{A'',\{(1,2)\}}, T_{A''',\{(1,2)\}}\in M'$ and also $T_{A'',\{(2,2)\}}\in M'$ because $A''$ is upper triangular, we get that $T_{A,\{(1,2),(2,1)\}}\in M'$.

Still assuming that $a_{12}\geq 1$, we induce again on $a_{21}\geq 2$  and using \ref{rem:anti-inv} and \ref{prop:comp-e-action}(e) we obtain
\begin{multline*}T_{A',\{(1,2),(2,1)\}}f_{a_{11}+a_{21}-1}=(e_{a_{11}+a_{21}-1}T_{^t A',\{(1,2),(2,1)\}})^\star=\\
=\vv^{2a_{21}+2a_{11}-2} T_{A'',\{(2,1)\}}+\vv^{2a_{21}+a_{11}-2}[a_{11}+1] T_{A'',\{(1,2),(2,1)\}}+\vv^{a_{21}-2}[a_{21}-1] T_{A,\{(1,2),(2,1)\}}.\end{multline*}
By previous computation, $T_{A'',\{(2,1)\}}\in M'$ and by induction $T_{A',\{(1,2),(2,1)\}},T_{A'',\{(1,2),(2,1)\}}\in M'$ so we can conclude that $T_{A,\{(1,2),(2,1)\}}\in M'$ for all $A$.
Finally, suppose $a_{22}\geq 1$, we then have by Prop. \ref{prop:comp-x-action}(g)
$$ T_{D(a_{11}+a_{12},a_{21}+a_{22}),\{(2,2)\}}T_{A,\emptyset}=T_{A,\{(2,1)\}}+T_{A,\{(1,2),(2,1)\}}+T_{A,\{(2,2)\}}.$$
Since $T_{A,\emptyset},T_{A,\{(2,1)\}},T_{A,\{(1,2),(2,1)\}}\in M'$, we can conclude that $T_{A,\{(2,2)\}}\in M'$.

Thus $T_{A,\Delta}\in M'$ for all $(A,\Delta)\in \Xi_{2,d}$ and $M'=MU_\vv(2,d)$.
\end{proof}

\begin{prop}\label{prop:rels}In the algebra $MU_\vv(2,d)$ we have the following relations among the Chevalley generators defined in the statement of Theorem \ref{theo:generators}:
\begin{align}
\label{rels:1}kk^{-1}&=1;\\
\label{rels:2}kek^{-1}&=\vv^2 e;\\
\label{rels:3}kfk^{-1}&=\vv^{-2} f;\\
\label{rels:4}ef-fe&=\df{k-k^{-1}}{\vv-\vv^{-1}};\\
\label{rels:5}\ell^2&=\ell;\\ 
\label{rels:6}k\ell&=\ell k;\\
\label{rels:7}\ell e\ell&=\ell e;\\
\label{rels:8}\ell f \ell& =f \ell;\\
\label{rels:9}[2] e\ell e&=\vv^{-1}e^2\ell+\vv \ell e^2;\\
\label{rels:10}[2] f\ell f&=\vv^{-1}\ell f^2+\vv f^2\ell.
\end{align}
\end{prop}
\begin{proof}
Relations \eqref{rels:1}-\eqref{rels:4} are the same as in the quantum Schur algebra and can be checked in the same way as in \cite{BLM}. For \eqref{rels:5} and \eqref{rels:6}, we just need to observe that $1_rx_s=x_s1_r=\delta_{rs}x_r$ and that, by Prop. \ref{prop:comp-x-action}(b), $x_r^2=(\vv^{2r}-2)x_r+(\vv^{2r}-1)1_r$.

To show \eqref{rels:7} and \eqref{rels:9} we use Propositions \ref{prop:comp-e-action} and \ref{prop:comp-x-action}. We compute 
$$e_rx_s=\delta_{rs}T_{D(r,d-r-1)+E_{1,2},\{(1,1)\}},$$ 
$$x_se_r=\delta_{s,r+1}(T_{D(r,d-r-1)+E_{1,2},\{(1,1)\}}+T_{D(r,d-r-1)+E_{1,2},\{(1,2)\}}),$$ 
and $x_{r+1}e_rx_r=(\vv^{2r}-1)e_r+(\vv^{2r}-2)T_{D(r,d-r-1)+E_{1,2},\{(1,1)\}}+(\vv^{2r}-1)T_{D(r,d-r-1)+E_{1,2},\{(1,2)\}}$. Hence
$$ \ell e \ell=\sum_{r=0}^{d-1}\vv^{-3r-2}\left(e_r+T_{D(r,d-r-1)+E_{1,2},\{(1,1)\}}+T_{D(r,d-r-1)+E_{1,2},\{(1,2)\}}\right)=\ell e.$$

We also have $e_se_r=\delta_{s,r+1}\vv[2] T_{D(r,d-r-2)+2E_{1,2},\emptyset}$, which implies that
$$ e^2 \ell =[2]\sum_{r=0}^{d-2}\vv^{-4r}\left(T_{D(r,d-r-2)+2E_{1,2},\emptyset}+T_{D(r,d-r-2)+2E_{1,2},\{(1,1)\}}\right);$$
and
$$ \ell e^2=[2]\sum_{r=0}^{d-2}\vv^{-4r-4}\left(T_{D(r,d-r-2)+2E_{1,2},\emptyset}+T_{D(r,d-r-2)+2E_{1,2},\{(1,1)\}}+T_{D(r,d-r-2)+2E_{1,2},\{(1,2)\}}\right);$$
thus
\begin{align*} e\ell e &=\sum_{r=0}^{d-2}\vv^{-4r-3}\left((\vv^2+1)T_{D(r,d-r-2)+2E_{1,2},\emptyset}+(\vv^2+1)T_{D(r,d-r-2)+2E_{1,2},\{(1,1)\}}+T_{D(r,d-r-2)+2E_{1,2},\{(1,2)\}}\right)\\
&= \sum_{r=0}^{d-2}\vv^{-4r-1}\left(T_{D(r,d-r-2)+2E_{1,2},\emptyset}+T_{D(r,d-r-2)+2E_{1,2},\{(1,1)\}}\right)\\
&~+\sum_{r=0}^{d-2}\vv^{-4r-3}\left(T_{D(r,d-r-2)+2E_{1,2},\emptyset}+T_{D(r,d-r-2)+2E_{1,2},\{(1,1)\}}+T_{D(r,d-r-2)+2E_{1,2},\{(1,2)\}}\right)\\
&= \frac{\vv^{-1}}{[2]}e^2\ell+\frac{\vv}{[2]}\ell e^2.
\end{align*}
The computations for \eqref{rels:8} and \eqref{rels:10} are analogous, using Prop. \ref{prop:comp-f-action} instead of Prop. \ref{prop:comp-e-action}.
\end{proof}
\begin{rem}Notice that Relations \eqref{rels:9} and \eqref{rels:10} are similar to the quantum Serre relations of type $A$, except for the appearance of the factors $\vv$ and $\vv^{-1}$ on the right hand side. More interestingly, \eqref{rels:9} (resp. \eqref{rels:10}) imply that $e$ and $\ell$ (resp. $f$ and $\ell$) satisfy the quantum Serre relation of type B, i.e. in $MU_\vv(2,d)$ we have
$$ e^3\ell-[3]e^2\ell e+[3]e\ell e^2-\ell e^3=0,  \text{ resp. }f^3\ell-[3]f^2\ell f+[3]f\ell f^2-\ell f^3=0.$$
This is an interesting phenomenon, appearing also in the mirabolic Hecke algebra. In fact in \cite[Lemma~4.13]{R14} we can see that the extra idempotent generator satisfies a type B braid relation with the first simple reflection $T_1$, in addition to the other relation (27). At the moment there is not a conceptual explanation for why this should be the case (and why $\ell$ should play the role of both the positive and negative simple root at the same time), although it was observed in the introduction to \cite{AH} that it is not surprising to see combinatorics of type B/C in the mirabolic setting, as it has connections with the `exotic' setting of Kato \cite{K09}.
\end{rem}
\begin{prop}
The anti-automorphism $^\star$ of Remark \ref{rem:anti-inv} is given on the Chevalley generators  by the following formulas:
$$ e^\star=\vv^{-1}k^{-1}f=\vv f k^{-1};\qquad f^\star=\vv^{-1}ke=\vv e k;\qquad k^\star=k;\qquad (k^{-1})^\star= k^{-1};\qquad \ell^\star=\ell.$$ 
\end{prop}
\begin{proof}
First of all, notice that $^t D(r,d-r)=D(r,d-r)$ and $^t \{(1,1)\}=\{(1,1)\}$, hence $1_r^\star=1_r$ and $x_r^\star=x_r$, which implies that $k^\star=k$, $(k^{-1})^\star=k^{-1}$, and $\ell^\star=\ell$.
Then, since $^t \left(D(r,d-r-1)+E_{1,2}\right)=D(r,d-r-1)+E_{2,1}$, we have that $e_r^\star=f_r$ and $f_r^\star=e_r$. Finally,
$$ e^\star=\left(\sum_{r=0}^{d-1}\vv^{-r} e_r\right)^\star = \sum_{r=0}^{d-1}\vv^{-r}f_r = \sum_{r=0}^{d-1}\vv^{-1}\vv^{d-2r}\vv^{1+r-d}1_rf_r = \vv^{-1}k^{-1}f,$$
and
$$ f^\star=\left(\sum_{r=0}^{d-1}\vv^{1+r-d} f_r\right)^\star = \sum_{r=0}^{d-1}\vv^{1+r-d}e_r = \sum_{r=0}^{d-1}\vv\vv^{2r-d}\vv^{-r}e_r1_r = \vv e k.$$
The two remaining equalities follow from the relations of Prop. \ref{prop:rels}.
\end{proof}

\section{Mirabolic quantum $\mathfrak{sl}_2$.}\label{sec:mir-quantum-sl2}
The relations among the Chevalley generators found in Proposition \ref{prop:rels} are not a complete list of relations for $MU_\vv(2,d)$ because there are also a lot of relations that depend on $d$, for example $e^{d+1}=0$ and $f^{d+1}=0$. Those extra relations can be hard to determine completely, therefore for now we will not focus on them and consider the algebra where no other relations appear.

Recall that the quantum enveloping algebra $\uvsl$ is the unital $\C(\vv)$-algebra with generators $e,f,k,k^{-1}$ satisfying relations \eqref{rels:1}-\eqref{rels:4}.

\begin{defin}The unital $\C(\vv)$-algebra with generators $e,f,k,k^{-1},\ell$ satisfying the relations of Prop. \ref{prop:rels} is called \emph{mirabolic quantum $\mathfrak{sl}_2$} and we denote it by $MU_\vv(2)$.\end{defin}

The relationship between $MU_\vv(2)$ and $MU_\vv(2,d)$ is analogous to the relationship between $U_\vv(\mathfrak{sl_2})$ and the quantum Schur algebra $\mathcal{S}_\vv(2,d)$ or Remark \ref{rem:inclusion}, in fact we have a commutative diagram:
\begin{center}
\begin{tikzpicture}
\draw (-1,0) node {$U_\vv(\mathfrak{sl}_2)$};
\draw (2.3,0) node {$MU_\vv(2)$};
\draw (-1,-2) node {$\mathcal{S}_\vv(2,d)$};
\draw (2.4,-2) node {$MU_\vv(2,d)$};
\draw[->>] (-1,-0.3) to (-1,-1.7);
\draw[->>] (2.2,-0.3) to (2.2,-1.7);
\draw[right hook->] (-0.3,0) to (1.5,0);
\draw[right hook->] (-0.3,-2) to (1.5,-2);
\end{tikzpicture}
\end{center}
Denote the inclusion $\iota:U_\vv(\mathfrak{sl}_2)\to MU_\vv(2)$ and notice that we also have two projections 
\begin{equation}\label{eq:2proj} U_\vv(\mathfrak{sl}_2)\stackrel{\pi_0}{\lonto} MU_\vv(2) \stackrel{\pi_1}{\ronto}U_\vv(\mathfrak{sl}_2)\end{equation}
where the maps take the Chevalley generators to the corresponding generators of $U_\vv(\mathfrak{sl}_2)$ and in addition we take $\pi_0(\ell)=0$ and $\pi_1(\ell)=1$. It is easy to check from the relations in Prop. \ref{prop:rels} that this gives a well defined map. 
\subsection{PBW Basis}
\begin{rem}\label{rem:anti-auto}
Notice that the relations in Proposition \ref{prop:rels} imply that the map defined on the Chevalley generators by 
$$e\mapsto f, \qquad f\mapsto e,\qquad k\mapsto k,\qquad \ell\mapsto\ell$$
 is an antiautomorphism of the algebra $MU_\vv(2)$.
\end{rem}
\begin{lem}\label{lem:move-out}
For each $a,b\in\N$, we have the following identities in $MU_\vv(2)$:
\begin{align*}
e^a\ell e^b&=\frac{\vv^{-b}[a]}{[a+b]}e^{a+b}\ell+\frac{\vv^{a}[b]}{[a+b]}\ell e^{a+b};\\
f^a\ell f^b&=\frac{\vv^{b}[a]}{[a+b]}f^{a+b}\ell+\frac{\vv^{-a}[b]}{[a+b]}\ell f^{a+b}.
\end{align*}
\end{lem}
\begin{proof}
By Remark \ref{rem:anti-auto}, it is enough to prove the first equality and the second one will follow by applying the antiautomorphism.
We use induction. The case $a=b=1$ is immediate from \eqref{rels:9}. Now suppose $b=1$ and induct on $a$:
\begin{align*} e^{a+1}\ell e&= e^a(e\ell e)=e^a\left(\frac{\vv^{-1}}{[2]}e^2\ell +\frac{\vv}{[2]}\ell e^2\right)\\
&= \frac{\vv^{-1}}{[2]}e^{a+2}\ell+\frac{\vv}{[2]}(e^a\ell e)e \\
\text{( by ind. hyp. )}&= \frac{\vv^{-1}}{[2]}e^{a+2}\ell+\frac{\vv}{[2]}\left(\frac{\vv^{-1}[a]}{[a+1]}e^{a+1}\ell +\frac{\vv^a}{[a+1]}\ell e^{a+1}\right)e \\
&=  \frac{\vv^{-1}}{[2]}e^{a+2}\ell+\frac{[a]}{[2][a+1]}e^{a+1}\ell e +\frac{\vv^{a+1}}{[2][a+1]}\ell e^{a+2} \\
\left(1-\frac{[a]}{[2][a+1]}\right)e^{a+1}\ell e&= \frac{\vv^{-1}}{[2]}e^{a+2}\ell+\frac{\vv^{a+1}}{[2][a+1]}\ell e^{a+2}\\
\frac{[a+2]}{[2][a+1]}e^{a+1}\ell e&= \frac{\vv^{-1}}{[2]}e^{a+2}\ell+\frac{\vv^{a+1}}{[2][a+1]}\ell e^{a+2}\\
e^{a+1}\ell e&= \frac{\vv^{-1}[a+1]}{[a+2]}e^{a+2}\ell+\frac{\vv^{a+1}}{[a+2]}\ell e^{a+2}.
\end{align*}
For general $b$, we have
\begin{align*}e^a\ell e^{b+1}&=(e^a\ell e^b)e=\left(\frac{\vv^{-b}[a]}{[a+b]}e^{a+b}\ell+\frac{\vv^{a}[b]}{[a+b]}\ell e^{a+b}\right)e\\
&=\frac{\vv^{-b}[a]}{[a+b]}e^{a+b}\ell e+\frac{\vv^{a}[b]}{[a+b]}\ell e^{a+b+1}\\ 
\text{( by ind. hyp. )}&=\frac{\vv^{-b}[a]}{[a+b]}\left(\frac{\vv^{-1}[a+b]}{[a+b+1]}e^{a+b+1}\ell+\frac{\vv^{a+b}}{[a+b+1]}\ell e^{a+b+1}\right)+\frac{\vv^{a}[b]}{[a+b]}\ell e^{a+b+1}\\
&= \frac{\vv^{-b-1}[a]}{[a+b+1]}e^{a+b+1}\ell+\left(\frac{\vv^a[a]}{[a+b][a+b+1]}+\frac{\vv^a[b]}{[a+b]} \right)\ell e^{a+b+1} \\
&=  \frac{\vv^{-b-1}[a]}{[a+b+1]}e^{a+b+1}\ell+\frac{\vv^{a}[b+1]}{[a+b+1]}\ell e^{a+b+1}.
\end{align*}
\end{proof}
\begin{prop}\label{prop:PBW}Consider the following collections of elements of $MU_\vv(2)$:   
\begin{align*}\mathscr{B}_0&= \{ f^re^sk^t~|~r,s\geq 0,~ t\in\Z\}, \\
\mathscr{B}_1&=\{\ell f^re^sk^t~|~r,s\geq0,~t\in\Z\}, \\
\mathscr{B}_2&=\{f^re^s\ell k^t~|~r,s\geq 0,~(r,s)\neq (0,0),~t\in Z\}, \\
\mathscr{B}_3&=\{\ell f^re^s\ell k^t~|~r,s\geq 1,~t\in\Z\},\\
\mathscr{B}_4&=\{f^r\ell e^s k^t~|~r,s\geq 1,~t\in\Z\},\\
\mathscr{B}_5&=\{e^s\ell f^rk^t~|~r,s\geq 1~t\in\Z\}.
\end{align*}
Then $\mathscr{B}=\bigsqcup\limits_{i=0}^5 \mathscr{B}_i$ spans $MU_\vv(2)$ over $\C(\vv)$.
\end{prop}
\begin{proof}
We show that the span of $\mathscr{B}$ is invariant under left multiplication by all the generators $e$, $f$, $k$, $k^{-1}$ and $\ell$, which implies the result.
It is immediate that $\mathscr{B}$ is invariant under multiplication by $k$ and $k^{-1}$, because of relations \eqref{rels:2}, \eqref{rels:3} and \eqref{rels:6}. Then, $\ell(\mathscr{B}_i)\subseteq\mathscr{B}_i$ for $i=1,3,4$ by \eqref{rels:5} and \eqref{rels:8}. Clearly $\ell(\mathscr{B}_0)\subseteq\mathscr{B}_1$. If $r,s\geq 1$, then $\ell (f^re^s\ell k^t) \in\sB_3$, while $\ell (f^r\ell k^t)=f^r\ell k^t\in\sB_2$ by \eqref{rels:8} and $\ell (e^s \ell k^t)=\ell e^s k^t\in \sB_1$ by \eqref{rels:7}. Finally, $\ell(e^s\ell f^r k^t)=\ell e^s f^r k^t$ which can be checked to be in the span of $\sB_1$ by the standard arguments of moving the $e$'s past the $f$'s using repeatedly \eqref{rels:4}. Multiplication by $e$ and $f$ can be handled in some cases using \eqref{rels:4} as in the case of $\uvsl$ (see for example \cite[$\S$ 1.3]{J}), but in other cases it is necessary to also use Lemma \ref{lem:move-out}. We will just give one example, the rest of the cases are very similar and will be omitted.
\begin{align*}e(\ell f e^r)&= e \ell \left(e^r f -[r] e^{r-1}\frac{k \vv^{r-1}-k^{-1}\vv^{1-r}}{\vv-\vv^{-1}}\right) \\
&= e\ell e^r f -[r] e \ell e^{r-1}\frac{k \vv^{r-1}-k^{-1}\vv^{1-r}}{\vv-\vv^{-1}} \\
&= \left(\frac{\vv^{-r}}{[r+1]}e^{r+1}\ell+\frac{\vv[r]}{[r+1]}\ell e^{r+1}\right)f-[r] \left(\frac{\vv^{1-r}}{[r]}e^{r}\ell+\frac{\vv[r-1]}{[r]}\ell e^{r}\right)\frac{k \vv^{r-1}-k^{-1}\vv^{1-r}}{\vv-\vv^{-1}},
\end{align*}
where in the first equality we used \cite[$\S$ 1.3 (6)]{J} and in the third equality we used Lemma \ref{lem:move-out}.
Now notice that the monomials in the generators appearing are 
$$e^{r+1}\ell f,\quad \ell e^{r+1} f,\quad e^r\ell k,\quad e^r \ell k^{-1},\quad \ell e^r k,\quad \ell e^r k^{-1},$$
which are all in $\sB$, except for $\ell e^{r+1}f$, for which we first need to use \cite[$\S$ 1.3 (6)]{J} one more time to get elements in the span of $\sB_1$. 
\end{proof}
To conclude that $\sB$ is a basis for $MU_\vv(2)$ we need to prove linear independence. To accomplish that we first need a partial order.
\begin{defin}
Let $A= \left(\begin{smallmatrix} a_{11} & a_{12} \\ a_{21} & a_{22} \end{smallmatrix}\right)$ and $B= \left(\begin{smallmatrix} b_{11} & b_{12} \\ b_{21} & b_{22} \end{smallmatrix}\right)$
such that $A,B\in\Theta_{2,d}$, $\co(A)=\co(B)$, $\ro(A)=\ro(B)$. We set $ A\leq B\text{ if }a_{12}\leq b_{12}\text{ and }a_{21}\leq b_{21}.$
We also define a partial order $\sqsubseteq$ on the set $\{\emptyset, \{(1,1)\},\{(1,2)\},\{(2,1)\},\{(1,2),(2,1)\},\{(2,2)\}\}$ by the following Hasse diagram (largest element on top)

\begin{center} 
\begin{tikzpicture}[yscale=0.5]
\draw (0,4) node {$\{(2,2)\}$};
\draw (0,2) node {$\{(1,2),(2,1)\}$};
\draw (1,0) node {$\{(2,1)\}$};
\draw (-1,0) node {$\{(1,2)\}$};
\draw (0,-2) node {$\{(1,1)\}$};
\draw (0,-4) node {$\emptyset$};
\draw[-] (0,2.5) to (0,3.5) ;
\draw[-] (1,0.5) to (0.2,1.5) ;
\draw[-] (-1,0.5) to (-0.2,1.5) ;
\draw[-] (1,-0.5) to (0.2,-1.5) ;
\draw[-] (-1,-0.5) to (-0.2,-1.5) ;
\draw[-] (0,-2.5) to (0,-3.5) ;
\end{tikzpicture}
\end{center}
For $(A,\Delta), (B,\Gamma)\in\Xi_{2,d}$, $\co(A)=\co(B)$, $\ro(A)=\ro(B)$ we say  
\begin{equation}\label{eq:partial-order} (A,\Delta)\preceq (B,\Gamma)\text{ if }A\leq B\text{ and }\Delta\sqsubseteq \Gamma\end{equation}
it is clear that this is a partial order. If $(A,\Delta)\preceq (B,\Gamma)$ and $(A,\Delta)\neq (B,\Gamma)$ we write $(A,\Delta)\prec (B,\Gamma)$.
\end{defin}
\begin{rem}The partial order $\leq$ on $\Theta_{2,d}$ is the same as the partial order given by orbit closures in $G_d\backslash\cF(2,d)\times \cF(2,d)$ but $\preceq$ is different from the order defined by orbit closures in $G_d\backslash\cF(2,d)\times\cF(2,d)\times\F_q^d$, which was described combinatorially in \cite{M}. For example, let $(A,\Delta)=\begin{pmatrix} 2 & 0 \\ 0 & \enc{1} \end{pmatrix}$ and $(B,\Gamma)=\begin{pmatrix} 1 & \enc{1} \\ \enc{1} & 0 \end{pmatrix}$. Then 
$\co(A)=\co(B)$ and $\ro(A)=\ro(B)$ and $\mathcal{O}_{A,\Delta}\subseteq \overline{\mathcal{O}_{B,\Gamma}}$. However $(A,\Delta)$ and $(B,\Gamma)$ are not comparable with respect to $\preceq$ because $A\leq B$ but $\{(2,2)\}\not\sqsubseteq\{(1,2),(2,1)\}$. 

\end{rem}
\begin{theo}The set $\sB$ is linearly independent over $\C(\vv)$, hence it is a basis of $MU_\vv(2)$. We call $\sB$ \emph{the PBW basis} of $MU_\vv(2)$.
\end{theo}
\begin{proof}Remember that for all $d\in\N_+$ we have the quotient map $MU_\vv(2)\ronto MU_\vv(2,d)$. Suppose that we have any finite set $B'=\{b_1,\ldots,b_p\}\subseteq\sB$, and let $R$ and $S$ be respectively the largest power of $f$ and $e$ appearing among the $b_i$'s. We want to show that there exists a large enough $d$ such that the images of $B'$ in $MU_\vv(2,d)$ are linearly independent, which will give the result. 

Now, suppose that $d>R+S$ and let $0\leq r\leq R$, $0\leq s\leq S$. We express products of the Chevalley generators of $MU_\vv(2,d)$ in terms of the basis $\{T_{A,\Delta}~|~(A,\Delta)\in\Xi_{2,d}\}$.
First of all, notice that by Propositions \ref{prop:comp-e-action} and \ref{prop:comp-f-action}, we have that
\begin{equation}\label{eq:es-fr} e^s=\sum_{t=0}^{d-s}\vv^{\beta(s,t)}T_{D(t,d-t-s)+sE_{1,2},\emptyset}\text{ and }f^r=\sum_{t=0}^{d-r}\vv^{\beta'(r,t)}T_{D(t,d-t-r)+rE_{2,1},\emptyset};\end{equation}
for some exponents $\beta(s,t), \beta'(r,t)\in\Z$.
By applying repeatedly \ref{prop:comp-f-action} to the first expression in \eqref{eq:es-fr} we obtain
\begin{equation}\label{eq:fres}f^re^s=\sum_{t=r}^{d-s}\left(\alpha(\vv) T_{D(t-r,d-t-s)+sE_{1,2}+rE_{2,1},\emptyset}+\text{ lower terms}\right);
\end{equation}
where $0\neq\alpha(\vv)\in\C(\vv)$. Here, and in what follows, whenever we write an expression like
$(T_{B,\Gamma}+\text{ lower terms })$ we mean that the lower terms are linear combinations of some $T_{A_i,\Delta_i}$ with $(A_i,\Delta_i)\prec (B,\Gamma)$. Notice that using this convention we can say that
$$ \ell=1_0 +\sum_{t=1}^d \left( \vv^{-2t}T_{D(t,d-t),\{(1,1)\}}+\text{ lower terms }\right).$$
From \eqref{eq:fres} and Prop. ~\ref{prop:comp-x-action}, we obtain
\begin{equation}
\label{eq:lfres}\ell f^re^s=\sum_{t=r}^{d-s}\left(\alpha'(\vv) T_{D(t-r,d-t-s)+sE_{1,2}+rE_{2,1},\{(1,2)\}}+\text{ lower terms}\right);
\end{equation}
for some $0\neq\alpha'(\vv)\in\C(\vv)$. Using the anti-involution $^\star$ and Prop.~\ref{prop:comp-x-action} to compute the products of $\ell$ on the right we also find that there are nonzero $\alpha''(\vv),\gamma(\vv)\in\C(\vv)$ such that
\begin{align}
\label{eq:fresl} f^re^s\ell&=\sum_{t=r}^{d-s}\left(\alpha''(\vv) T_{D(t-r,d-t-s)+sE_{1,2}+rE_{2,1},\{(2,1)\}}+\text{ lower terms}\right);\\
\label{eq:lfresl} \ell f^re^s\ell&=\sum_{t=r}^{d-s}\left(\gamma(\vv) T_{D(t-r,d-t-s)+sE_{1,2}+rE_{2,1},\{(1,2),(2,1)\}}+\text{ lower terms}\right).
\end{align}
Notice also that by \eqref{eq:es-fr} and Prop.~\ref{prop:comp-x-action} we obtain that
\begin{align}
\label{eq:les} \ell e^s&=\sum_{t=0}^{d-s}\left(\vv^{\beta''} T_{D(t,d-t-s)+sE_{1,2},\{(1,2)\}}+\text{ lower terms}\right);\\
\label{eq:lfr} \ell f^r&=\sum_{t=0}^{d-r}\left(\vv^{\beta'''}T_{D(t,d-t-r)+rE_{2,1},\{(1,1)\}}+\text{ lower terms}\right).
\end{align}
It follows then from \eqref{eq:les} and \eqref{eq:lfr} by applying several times Propositions \ref{prop:comp-e-action} and \ref{prop:comp-f-action} that for some nonzero $\gamma'(\vv),\gamma''(\vv),\gamma'''(\vv)\in\C(\vv)$ we have
\begin{align}
\label{eq:eslfr} e^s\ell f^r&=\sum_{t=r}^{d-s}\left(\gamma'(\vv) T_{D(t-r,d-t-s)+sE_{1,2}+rE_{2,1},\{(1,1)\}}+\text{ lower terms}\right). \\
\label{eq:frles} f^r\ell e^s&=\sum_{t=r}^{d-s}\left(\gamma''(\vv) T_{D(t-r,d-t-s)+sE_{1,2}+rE_{2,1},\{(2,2)\}}+\text{ lower terms}\right)+\\ 
\notag &+\sum_{t=r-1}^{d-s+1}\left(\gamma'''(\vv)T_{D(t-r+1,d-t-s+1)+(s-1)E_{1,2}+(r-1)E_{2,1},\{(1,2),(2,1)\}}+\text{ lower terms }\right);
\end{align}
Equations \eqref{eq:fres}-\eqref{eq:frles} prove that if 
$$B'\subseteq \{f^re^s,\ell f^re^s,f^re^s\ell,\ell f^re^s\ell, f^r\ell e^s, e^s\ell f^r~|~r,s\text{ as in Prop.~\ref{prop:PBW}},r\leq R,~s\leq S \}$$ then $B'$ is a linearly independent set. But then, from the expression of these monomials in terms of the basis $\{T_{A,\Delta}~|~(A,\Delta)\in\Xi_{2,d}\}$ it is also clear, by the Vandermonde determinant, that right multiplication by a power of $k$ also yields linearly independent terms (possibly by taking a bigger $d$) which concludes the proof.
\end{proof}

\section{Representations of $MU_\vv(2)$}\label{sec:mir-reps}
Representations of Lie algebras and their quantum analogues are studied using the weight decomposition for the action of the Cartan subalgebra. In the case of $U_\vv(\mathfrak{sl}_2)$, this corresponds to studying the eigenspaces for the action of $k$. In the case of $MU_\vv(2)$, since the elements $k$ and $\ell$ commute, we can consider the decomposition of representations of $MU_\vv(2)$ into simultaneous eigenspaces for $k$ and $\ell$. Notice that the only possible eigenvalues of $\ell$ are $0$ and $1$, because it is an idempotent.

\begin{defin}If $V$ is a left module for $MU_\vv(2)$, $\lambda\in\C(\vv)$, $\epsilon\in\{0,1\}$, we define the \emph{weight space}  $V_{\lambda,\epsilon}=\{v\in V~|~ kv=\lambda v, ~\ell v=\epsilon v\}$. If $(\lambda,\epsilon)$ is such that $V_{\lambda,\epsilon}\neq 0$, we say that $(\lambda,\epsilon)$ is a \emph{weight} of $V$. We say that $V$ is a \emph{weight} module for $MU_\vv(2)$ if $V=\oplus_{\lambda,\epsilon} V_{\lambda,\epsilon}$.
\end{defin}
\begin{rem}\label{rem:weight-action}By  relations \eqref{rels:7} and \eqref{rels:8} we get that for any $MU_\vv(2)$-module $V$, $\ker\ell$ is invariant under $e$ and $\im\ell$ is invariant under $f$, in fact for all $v\in\ker\ell$ and $w\in\im\ell$ we have
$$ \ell (e v)=\ell e\ell v=\ell e (0)=0\qquad\text{ and }\qquad\ell(f w)=\ell f (\ell w)=f\ell (w)=fw.$$
It then follows from \eqref{rels:2} and \eqref{rels:3} that $e(V_{\lambda,0})\subseteq V_{\vv^{2}\lambda,0}$ and $f(V_{\lambda,1})\subseteq V_{\vv^{-2}\lambda,1}$. We also know that $e(V_{\lambda,1})\subseteq (V_{\vv^{2}\lambda,0}\oplus V_{\vv^2\lambda,1})$ and $f(V_{\lambda,0})\subseteq (V_{\vv^{-2}\lambda,0}\oplus V_{\vv^{-2}\lambda,1})$.
\end{rem}
\begin{prop}\label{prop:weight-mod}Let $V$ be a finite dimensional $MU_\vv(2)$ module, then $V$ is a weight module and all the weights are of the form $(\pm \vv^a,\epsilon)$ with $a\in\Z$ and $\epsilon\in\{0,1\}$.
\end{prop}
\begin{proof}Using the inclusion $\iota:U_\vv(\mathfrak{sl}_2)\rinto MU_\vv(2)$, $V$ becomes a finite dimensional $U_\vv(\mathfrak{sl}_2)$-module. Hence by \cite[2.3]{J} it is the direct sum of its weight spaces for the action of $k$ with weights $\pm \vv^a$. The statement then follows because $\ell$ is an idempotent that commutes with $k$.
\end{proof}
If $V$ is a module for $U_\vv(\mathfrak{sl}_2)$, we get two modules for $MU_\vv(2)$, $\pi_0^*(V)$ and $\pi_1^*(V)$, given by pullback along the projections of \eqref{eq:2proj}. By definition, $\ell$ acts as zero (resp. the identity) on $\pi_0^*(V)$ (resp. $\pi_1^*(V)$). Conversely, if $V$ is a module for $MU_\vv(2)$ where $\im\ell\subseteq V$ and $\ker\ell\subseteq V$ are submodules, then we have an $MU_\vv(2)$-module decomposition
$$V\simeq\im\ell\oplus\ker\ell=\pi_1^*(V^1)\oplus\pi_0^*(V^0)$$ 
for some $U_\vv(\mathfrak{sl_2})$-modules $V^1$ and $V^0$.

We are especially insterested, then, in finding modules for $MU_\vv(2)$ where $\im\ell$ and $\ker\ell$ are not submodules.
\begin{prop}\label{prop:L01}
Let $n\in\N_+$, consider the $\C(\vv)$-vector spaces $L^+(n,01)$ and $L^-(n,01)$ with respective bases $\{m^\pm_{i,0}~|~0\leq i\leq n-1\}\cup\{m^\pm_{j,1}~|~1\leq j\leq n\}$.
Then the following maps make $L^\pm(n,01)$ into $MU_\vv(2)$-modules.
\begin{align}
\label{eq:k-action} k\cdot m^\pm_{i,\epsilon}&=\pm\vv^{n-2i}m^\pm_{i,\epsilon} \\
\label{eq:l-action}\ell\cdot m^\pm_{i,\epsilon}&=\epsilon m^\pm_{i,\epsilon} \\
\label{eq:f-action0}f\cdot m^\pm_{i,0}&=m^\pm_{i+1,0}+\frac{\vv^i}{[i+1]}m^\pm_{i+1,1} \\
\label{eq:f-action1}f\cdot m^\pm_{i,1}&= \vv^{-1}\frac{[i]}{[i+1]}m^\pm_{i+1,1} \\
\label{eq:e-action0}e\cdot m^\pm_{i,0}&= \pm\vv[i][n-i] m^\pm_{i-1,0} \\
\label{eq:e-action1}e\cdot m^\pm_{i,1}&= \pm[i][n+1-i] m^\pm_{i-1,1}\pm\vv^{i-n}[i] m^\pm_{i-1,0}
\end{align}  
Here we interpret $m^\pm_{i,\epsilon}$ as zero, if the index $i$ does not satisfy the conditions in the definition of the bases.
\end{prop}
\begin{proof}We check that relations \eqref{rels:2}-\eqref{rels:10} are satisfied in the case of $L^+(n,01)$, the case of $L^-(n,01)$ then follows directly.
Observe that \eqref{rels:2}, \eqref{rels:3}, \eqref{rels:5}, \eqref{rels:6}, \eqref{rels:7} and \eqref{rels:8} are immediate using Remark \ref{rem:weight-action} because by \eqref{eq:k-action} and \eqref{eq:l-action} the basis element $m^+_{i,\epsilon}$ is in the $(\vv^{n-2i},\epsilon)$ weight space of the module. 

To check \eqref{rels:4}, compute
\begin{align*}ef(m^+_{i,0})&=e\left(m^+_{i+1,0}+\frac{\vv^i}{[i+1]}m^+_{i+1,1}\right)\\
&= \vv[i+1][n-i-1] m^+_{i,0}+ \vv^i[n-i] m^+_{i,1}+\vv^{i+1-n}[i+1] m^+_{i,0};
\end{align*}
$$ fe(m^+_{i,0})=f( \vv[i][n-i] m^+_{i-1,0})= \vv[i][n-i] m^+_{i,0}+\vv^i[n-i] m^+_{i,1};$$
hence
\begin{align*} (ef-fe)m^+_{i,0}&= (\vv[i+1][n-i-1] +\vv^{i+1-n}[i+1]- \vv[i][n-i]) m^+_{i,0} \\
&= [n-2i] m^+_{i,0} = \df{k-k^{-1}}{\vv-\vv^{-1}}m^+_{i,0}.
\end{align*}
We also need to compute

$$ ef(m^+_{i,1})=e\left( \vv^{-1}\frac{[i]}{[i+1]}m^+_{i+1,1}\right)= \vv^{-1}[i][n-i] m^+_{i,1}+\vv^{i-n}[i] m^+_{i,0};$$
\begin{align*}fe(m^+_{i,1})&=f([i][n+1-i] m^+_{i-1,1}+\vv^{i-n}[i] m^+_{i-1,0})\\
&= \vv^{-1}[i-1][n+1-i] m^+_{i,1}+\vv^{i-n}[i] m^+_{i,0}+\vv^{2i-n-1}m^+_{i,1};
\end{align*}
thus
\begin{align*}(ef-fe)m^+_{i,1}&=(\vv^{-1}[i][n-i]- \vv^{-1}[i-1][n+1-i]-\vv^{2i-n-1})m^+_{i,1}\\
&= [n-2i] m^+_{i,1}= \df{k-k^{-1}}{\vv-\vv^{-1}}m^+_{i,1}.
\end{align*}
Now, we want to check \eqref{rels:9}. One case is very simple, since $e$ preserves $\ker\ell$, then
$$[2] e\ell e(m^+_{i,0})=0=(\vv^{-1}e^2\ell+\vv \ell e^2)m^+_{i,0}.$$
For the other case we compute
$$e^2(m^+_{i,1})=[i][i-1][n+1-i][n+2-i] m^+_{i-2,1}+(\vv^{i+1-n}+\vv^{i-1-n})[i][i-1][n+1-i]m_{i-2,0};$$
hence
\begin{align*}[2] e\ell e(m^+_{i,1})=&(\vv+\vv^{-1})[i][i-1][n+1-i][n+2-i] m^+_{i-2,1}+\\
&+(\vv+\vv^{-1})\vv^{i-1-n}[i][i-1][n+1-i]m_{i-2,0}\\
=& \vv[i][i-1][n+1-i][n+2-i] m^+_{i-2,1}+\\
&+\vv^{-1}[i][i-1][n+1-i][n+2-i] m^+_{i-2,1}+\\
&+\vv^{-1}(\vv^{i+1-n}+\vv^{i-1-n})[i][i-1][n+1-i]m_{i-2,0}\\
=& \vv \ell e^2 (m^+_{i,1})+\vv^{-1}e^2\ell(m^+_{i,1}).
\end{align*}
Finally, since $f$ preserves $\im\ell$, we have
$$[2] f\ell f(m^+_{i,1})=(\vv+\vv^{-1})\vv^{-2}\frac{[i]}{[i+2]}m^+_{i+2,1}=(\vv^{-1}\ell f^2+\vv f^2\ell)m^+_{i,1};$$
and we compute
$$ f^2(m^+_{i,0})=m^+_{i+2,0}+(\vv+\vv^{-1})\frac{\vv^i}{[i+2]}m^+_{i+2,1};$$
therefore
\begin{align*}[2] f\ell f(m^+_{i,0})&=(\vv+\vv^{-1})\frac{\vv^{i-1}}{[i+2]}m^+_{i+2,1} \\
&= \vv^{-1}\ell f^2 (m^+_{i,0})+\vv f^2\ell(m^+_{i,0})
\end{align*}
which shows that \eqref{rels:10} is satisfied and concludes the proof.
\end{proof}

\begin{prop}For all $n\in\N_+$, the $MU_\vv(2)$-modules $L^+(n,01)$ and $L^-(n,01)$ are simple.
\end{prop}
\begin{proof}
Suppose $0\neq M'\subseteq L^{\pm}(n,01)$ is a submodule. Since $M'$ is invariant under the action of $k$ and $\ell$, it is a weight module, hence there is a pair $(i,\epsilon)$ such that $m^{\pm}_{i,\epsilon}\in M'$. Say $\epsilon=0$, then by \eqref{eq:e-action0} and \eqref{eq:f-action0} we know that $em^{\pm}_{i,0}$ is some nonzero multiple of $m^{\pm}_{i-1,0}$, $\ell f m^{\pm}_{i,0}$ is a nonzero multiple of $m^{\pm}_{i+1,1}$ and $(1-\ell)fm^{\pm}_{i,0}$ is a nonzero multiple of $m^{\pm}_{i+1,0}$. Hence $m^{\pm}_{i-1,0}$, $m^\pm_{i+1,0}$, $m^\pm_{i+1,1}\in M'$. Analogously, assume that there is a $j$ such that $m^{\pm}_{j,1}\in M'$, then by \eqref{eq:f-action1} and \eqref{eq:e-action1} we deduce that  $m^{\pm}_{j+1,1}$, $m^\pm_{j-1,0}$, $m^\pm_{j-1,1}\in M'$.
Iterating this argument, since all the coefficients appearing in the action by $e$ and $f$ are nonzero, we obtain that if  $m^{\pm}_{i,\epsilon}\in M'$ for any $(i,\epsilon)$, then $m^{\pm}_{i,\epsilon}\in M'$ for all $(i,\epsilon)$ hence $M'=L^{\pm}(n,01)$.
\end{proof}
\begin{defin}For all $n\in\N$, we let $L^\pm(n)$ be the simple $\uvsl$-module with highest weight $\pm\vv^{n}$, and we define 
$$ L^\pm(n,0):=\pi_0^*(L^\pm(n));\qquad L^\pm(n,1):=\pi_1^*(L^\pm(n)).$$
\end{defin}
\begin{prop}For all $n\in\N_+$, consider $L^\pm(n,01)$ as an $\uvsl$-module via the inclusion $\iota$, then we have the following isomorphism of $U_\vv(\mathfrak{sl}_2)$-modules:
$$ L^{\pm}(n,01)\simeq L^+(n-1)\otimes L^\pm(1).$$
In particular this means that, as $\uvsl$-modules, we have $L^\pm(1,01)\simeq L^\pm(1)$ and, for $n>1$, $L^\pm(n,01)\simeq L^\pm(n)\oplus L^\pm(n-2)$.
\end{prop}
\begin{proof}
This is immediate by the decomposition of $L^{\pm}(n,01)$ into weight spaces for the action of $k$. 
\end{proof}
\begin{theo}\label{thm:classification}The following is a complete list of pairwise non-isomorphic finite dimensional simple modules for $MU_\vv(2)$, up to isomorphism:
$$ \{L^\pm(n,0)~|~ n\in\N \}\cup\{L^\pm(n,1)~|~ n\in\N\}\cup\{L^\pm(n,01)~|~n\in\N_+\}.$$
\end{theo}
\begin{proof}
First of all, by checking the decomposition into weight spaces for the action of $k$ and $\ell$ it becomes clear that the modules in the list are all pairwise non-isomorphic. Now suppose that $M$ is a simple, finite dimensional, $MU_\vv(2)$-module and we want to show that it is isomorphic to one of the modules in our list. Since $M$ is finite dimensional, by Prop. \ref{prop:weight-mod} it is a weight module and the weights are all of the form $(\pm\vv^a,\epsilon)$ with $a\in\Z$ and $\epsilon\in\{0,1\}$. Since $M$ is finite dimensional, the set of weights of $M$ has to be finite, therefore there exists a highest weight $(\lambda_0,\epsilon_0)$ such that $M_{\lambda_0,\epsilon_0}\neq 0$ and $M_{\vv^2\lambda,0}=M_{\vv^2\lambda,1}=0$. Considering $M$ as a $\uvsl$-module via the inclusion $\iota$, we get that $\lambda_0=\pm\vv^n$ for some $n\in\N$. Fix a highest weight vector $0\neq v_0\in  M_{\lambda_0,\epsilon_0}$.
\subsubsection*{Case1: $\epsilon_0=1$}
Let $v_i=f^iv_0$. Since $f(\im\ell)\subseteq \im\ell$, we have $\ell v_i=v_i$ for all $i$. We also have $ev_0=0$, since $v_0$ is a highest weight vector, and 
$$ev_1=efv_0=\left(fe+\frac{k-k^{-1}}{\vv-\vv^{-1}}\right)v_0=\pm[n] v_0.$$ 
By induction, then, for all $i>1$ we have that 
$$ev_i=efv_{i-1}=\left(fe+\frac{k-k^{-1}}{\vv-\vv^{-1}}\right)v_{i-1}$$ 
is a multiple of $v_{i-1}$. Since $M$ is simple, $M=MU_\vv(2)\cdot v_0$, hence $M=\operatorname{span}\{ v_i\}_i$. It follows that $\ell|_M=\id_M$, hence $M=\pi_1^*(V^1)$ for some $\uvsl$-module $V^1$. Since $M$ is simple and has highest weight $(\pm\vv^n,1)$, we get that $V^1\simeq L^\pm(n)$, and $M\simeq L^\pm(n,1)$.
\subsubsection*{Case 2: $\epsilon_0=0$, $\ell f(v_0)=0$}
Again, let $v_i=f^iv_0$. Since by assumption $v_0, fv_0\in\ker\ell$, by induction we have, for all $i>2$,
$$ \ell v_i=\ell f^iv_0\stackrel{\eqref{rels:10}}{=}(\vv^2+1)f\ell f^{i-1}v_0-\vv^2 f^2\ell f^{i-2}v_0=0.$$
Hence, for all $i$, $v_i\in\ker\ell$ and by the same reasoning as in Case 1, $ev_i$ is a multiple of $v_{i-1}$. We can then deduce that $M=\operatorname{span}\{ v_i\}_i$ and that $\ell|_M=0$. In conclusion, $M=\pi_0^*(V^0)$ for some $\uvsl$-module $V^0$, and the only possibility is $V^0\simeq L^{\pm}(n)$. Therefore $M\simeq L^{\pm}(n,0)$.
\subsubsection*{Case 3: $\epsilon_0=0$, $\ell f(v_0)\neq 0$}
For all $i\geq 0$, let $v_i=f^i v_0$ and let $v_{i,1}=\ell v_i$, $v_{i,0}=(1-\ell)v_i$. Notice that $v_{0,0}=v_0$ and $v_{0,1}=0$. Consider $M'=\operatorname{span}\{v_{i,\sigma}~|~i\geq 0, ~\sigma=0,1\}$. Clearly $M'$ is invariant under the action of $k$, $k^{-1}$ and $\ell$. 
For all $i\geq 1$, we show by induction on $i$ that 
\begin{equation}\label{eq:ind-f}f v_{i,1}=\frac{\vv^{-1}[i]}{[i+1]}v_{i+1,1}.\end{equation}
For $i=1$, 
\begin{align*}f v_{1,1}=f\ell f v_{0,0}\stackrel{\eqref{rels:10}}{=}\left (\frac{\vv^{-1}}{[2]}\ell f^2+\frac{\vv}{[2]}f^2\ell \right)v_{0,0} = \frac{\vv^{-1}}{[2]}\ell f^2 v_{0,0}+0= \frac{\vv^{-1}}{[2]}v_{2,1}.
\end{align*}
For $i>1$, we have
\begin{align*}f v_{i,1}&=f\ell f^i v_{0,0}\stackrel{\eqref{rels:10}}{=}\left (\frac{\vv^{-1}}{[2]}\ell f^2+\frac{\vv}{[2]}f^2\ell \right)f^{i-1}v_{0,0}\\
&=\frac{\vv^{-1}}{[2]}\ell f^{i+1}v_{0,0}+\frac{\vv}{[2]}f(f\ell f^{i-1}v_{0,0})\\
\text{( by ind. hyp. )}&=\frac{\vv^{-1}}{[2]}v_{i+1,1}+\frac{\vv}{[2]}f\left(\frac{\vv^{-1}[i-1]}{[i]}v_{i,1}\right)\\
\left(1-\frac{[i-1]}{[2][i]}\right)f v_{i,1}&=\frac{\vv^{-1}}{[2]}v_{i+1,1}\\
f v_{i,1}&=\frac{\vv^{-1}[i]}{[i+1]}v_{i+1,1}.
\end{align*}
This now implies also that 
\begin{equation}\label{eq:ind-f0} f v_{i,0}=f(v_i-v_{i,1})=v_{i+1}-fv_{i,1}=v_{i+1,0}+v_{i+1,1}-\frac{\vv^{-1}[i]}{[i+1]}v_{i+1,1}=v_{i+1,0}+\frac{\vv^i}{[i+1]}v_{i+1,1},\end{equation}
which proves that $f(M')\subseteq M'$.
Now we want to prove that $M'$ is invariant under the action of $e$, which will show that $M'=M$.
In order to do that, we will first show that in this case $e v_{1,1}=e\ell f(v_0)$ is a multiple of $v_0$. 

Suppose by contradiction that $w_0:=e\ell f(v_0)$ and $v_0$ are linearly independent. Clearly $kw_0=\pm\vv^n w_0$. Write $w_0=w_{0,0}+w_{0,1}$ where $w_{0,\sigma}\in M_{\pm\vv^n,\sigma}$ for $\sigma=1,2$. Then $w_{0,1}=\ell w_0\in M_{\lambda_0,1}$. This implies that $w_{0,1}=0$ because otherwise, as proved in case $1$, $MU_\vv(2)\cdot w_{0,1}\simeq L^\pm(n,1)$ would be a nonzero proper submodule of $M$ ($v_0\not\in  MU_\vv(2)\cdot w_{0,1}$). So $w_0=e\ell f v_0\in M_{\pm\vv^n,0}$.
We then have
\begin{align} \notag \ell f w_0 &= \ell f e \ell f v_0 \\
\notag &= \ell (fe-ef+ef)\ell f v_0 \\
\notag &= \ell \left(\frac{k^{-1}-k}{\vv-\vv^{-1}}\right)\ell f v_0+\ell ef\ell f v_0 \\
\notag &= \mp[n-2] \ell f v_0+\frac{1}{[2]}\ell e \left(\vv^{-1}\ell f^2+\vv f^2\ell\right) v_0 \\
\notag &=  \mp[n-2] \ell f v_0+\frac{\vv^{-1}}{[2]}\ell e \ell f^2 v_0 \\
\notag &=  \mp[n-2] \ell f v_0+\frac{\vv^{-1}}{[2]}\ell e f^2 v_0 \\
\notag &=   \mp[n-2] \ell f v_0+\frac{\vv^{-1}}{[2]}\ell \left(\frac{k-k^{-1}}{\vv-\vv^{-1}}+fe\right) fv_0\\
\notag &=  \left( \mp[n-2]\pm\frac{\vv^{-1}}{[2]}[n-2] \right)\ell f v_0+\frac{\vv^{-1}}{[2]}\ell fef v_0 \\
\notag &=  \left( \mp[n-2]\pm\frac{\vv^{-1}}{[2]}[n-2]\right)\ell f v_0+\frac{\vv^{-1}}{[2]}\ell f\left(\frac{k-k^{-1}}{\vv-\vv^{-1}}+fe\right) v_0 \\
\notag &=  \left( \mp[n-2]\pm\frac{\vv^{-1}}{[2]}([n-2] +[n])\right)\ell f v_0 \\
\label{eq:lin-dip} \ell f w_0 &= \pm \vv^{-n+1} \ell f v_0. 
\end{align}
Then $u_0:=\pm \vv^{-n+1} v_0-w_0\in M_{\pm \vv^n,0}$ is such that $u_0\neq 0$ and, based on the above computation, $\ell f (u_0)=0$.
But, according to Case 2, this would imply that $MU_\vv(2)\cdot u_0\simeq L^\pm(n,0)$ is a proper submodule ($v_0\not\in MU_\vv(2)\cdot u_0$) of $M$ which is impossible.

Now, since $w_0=e\ell f v_0=e v_{1,1}$ is a multiple of $v_0$, \eqref{eq:lin-dip} implies that $e v_{1,1}=\pm\vv^{-n+1}v_0$.
Remark that 
$$e v_1=ef v_0=(ef-fe+fe)v_0=\frac{k-k^{-1}}{\vv-\vv^{-1}}v_0+0=\pm[n] v_0;$$ 
hence 
$$e v_{1,0}=e(v_1-v_{1,1})=\pm[n] v_0-(\pm\vv^{-n+1})v_0=\pm \vv[n-1] v_0.$$
By induction, we prove that for all $i\geq 2$
\begin{equation}\label{eq:ind-e} e v_{i,1}=\pm[i][n+1-i] v_{i-1,1}\pm\vv^{i-n}[i] v_{i-1,0}.\end{equation}
Base case
\begin{align*} e v_{2,1}&=e\ell f^2 v_0=e(\vv[2] f\ell f-\vv^2 f^2\ell)v_0 \\
&= \vv[2] (ef)\ell f v_0 \\
&= \vv[2]\left(\frac{k-k^{-1}}{\vv-\vv^{-1}}\right)\ell f v_0+\vv[2] f(e\ell f v_0) \\
&= \pm\vv[2][n-2] v_{1,1}\pm\vv[2] f(\vv^{-n+1} v_0)\\
&=\pm\vv[2][n-2] v_{1,1}\pm\vv^{-n+2}[2] v_{1,1}\pm\vv^{-n+2}[2] v_{1,0}\\
&= \pm[2][n-1] v_{1,1}\pm \vv^{-n+2}[2] v_{1,0}.
\end{align*}
In general, for $i\geq 2$, we have
\begin{align*}e v_{i,1}&=e\ell f^i v_0=e\ell f^2 v_0=e(\vv[2] f\ell f-\vv^2 f^2\ell)f^{i-2}v_0 \\
&=\vv[2](ef)\ell f^{i-1}v_0-\vv^2(ef^2)\ell f^{i-2}v_0\\
&=\vv[2]\left(\frac{k-k^{-1}}{\vv-\vv^{-1}}\right)\ell f^{i-1}v_0+\vv[2] f(e \ell f^{i-1}v_0)-\vv^2\left([2] f \frac{k\vv^{-1}-k^{-1}\vv}{\vv-\vv^{-1}}\right)\ell f^{i-2}v_0\\
&-\vv^2 (f^2 e)\ell f^{i-2}v_0 \\
\text{( by IH )}&=\pm\vv[2] [n-2i+2] v_{i-1,1}+\vv[2] f (\pm[i-1][n+2-i] v_{i-2,1}) \\
&+\vv[2] f(\pm\vv^{i-1-n}[i-1] v_{i-2,0})\pm\vv^2[2][n-2i+1] f v_{i-2,1}\\
&-\vv^2 f^2(\pm[i-2][n+3-i] v_{i-3,1}\pm\vv^{i-2-n}[i-2] v_{i-3,0})
\end{align*}
with a tedious computation, using \eqref{eq:ind-f} and \eqref{eq:ind-f0}, this last expression can be shown to be equal to
$$ \pm[i][n+1-i] v_{i-1,1}\pm\vv^{i-n}[i] v_{i-1,0},$$
which concludes the induction.

Notice that \eqref{eq:ind-e} also implies that, for all $i\geq 2$,
\begin{equation}\label{eq:ind-e0} e v_{i,0}=e (v_i-v_{i,1})=\pm[i][n+1-i] (v_{i-1,0}+v_{i-1,1})-e v_{i,1}=\pm\vv[i][n-i] v_{i,0}.\end{equation}
We therefore have that $M'=\operatorname{span}\{v_{i,\sigma}~|~i\geq 0, ~\sigma=0,1\}=M$. 

Since $M$ is finite dimensional, there is a $j\in\N$ such that $v_{i,\sigma}=0$ for all $\sigma$ and for all $i> j$; let $j_0$ be minimal with this property. It follows from the weight space decomposition of $M$ as an $\uvsl$-representation that $j_0=n$. Furthermore, from the same decomposition it follows that the eigenspace for $k$ with eigenvalue $\pm\vv^{-n}$ is one dimensional, hence exactly one between $v_{n,0}$ and $v_{n,1}$ is equal to zero. Suppose by contradiction that $v_{n,0}\neq 0$. Since $M$ is a simple $MU_\vv(2)$-module, we have that $M=MU_\vv(2)\cdot v_{n,0}$. But notice that $e (\ker \ell)\subseteq \ker \ell$, hence $\ell e^i v_{n,0}=0$ for all $i\geq 0$. With the same argument as in Case 1, it would follow that $\operatorname{span}\{e^iv_{n,0}\}_i=M$, because it is also invariant under $f$, but this is impossible because $M\not\subset\ker\ell$.
In conclusion we have that $v_{n,0}=0$ and $v_{n,1}\neq 0$ and, by comparing \eqref{eq:ind-f}, \eqref{eq:ind-f0}, \eqref{eq:ind-e} and \eqref{eq:ind-e0} with the formulae \eqref{eq:f-action0}-\eqref{eq:e-action1}, we have that $M\simeq L^{\pm}(n,01)$.
\end{proof}
\begin{exa}We can represent the weight space decomposition and the action of $f$ on the simple modules for $MU_\vv(2)$ in a diagram. In what follows the dots represent one dimensional spaces and are labelled by their weight, the arrows represent the action of $f$. 

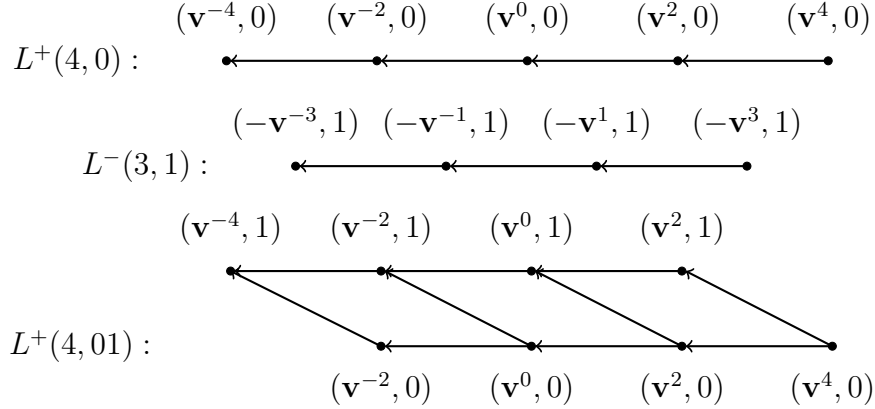
\begin{center}
\begin{tikzpicture}[xscale=0.5,yscale=0.5]
\draw (-12,0) node {$L^+(4,0):$};
\filldraw (8,0) circle(3pt);
\filldraw (4,0) circle(3pt);
\filldraw (0,0) circle(3pt);
\filldraw (-4,0) circle(3pt);
\filldraw (-8,0) circle(3pt);
\foreach \x in {-4,-2,0,2,4}
{\draw (2*\x,0.4) node[anchor=south] {$(\vv^{\x},0)$};}
\foreach \x in {-2,0,2,4}
{ \draw[->] (2*\x,0) to (2*\x-3.9,0) [thick];}
\end{tikzpicture}
\end{center}

\begin{center}
\begin{tikzpicture}[xscale=0.5,yscale=0.5]
\draw (-10,0) node {$L^-(3,1):$};
\filldraw (6,0) circle(3pt);
\filldraw (2,0) circle(3pt);
\filldraw (-2,0) circle(3pt);
\filldraw (-6,0) circle(3pt);
\foreach \x in {-3,-1,1,3}
{\draw (2*\x,0.4) node[anchor=south] {$(-\vv^{\x},1)$};}
\foreach \x in {-1,1,3}
{ \draw[->] (2*\x,0) to (2*\x-3.9,0) [thick];}
\end{tikzpicture}
\end{center}

\begin{center}
\begin{tikzpicture}[xscale=0.5,yscale=0.5]
\draw(-12,0) node {$L^+(4,01):$};
\foreach \x in {-3,-1,1,3}
{\filldraw (2*\x+2,0) circle(3pt);
\filldraw (2*\x-2,2) circle(3pt);
\draw (2*\x+2,-0.4) node[anchor=north] {$(\vv^{\pgfmathparse{\x+1}\pgfmathprintnumber{\pgfmathresult}},0)$};
\draw (2*\x-2,2.4) node[anchor=south] {$(\vv^{\pgfmathparse{\x-1}\pgfmathprintnumber{\pgfmathresult}},1)$};
}
\foreach \x in {0,2,4}
{ \draw[->] (2*\x,0) to (2*\x-3.9,0) [thick];
\draw[->] (2*\x,0) to (2*\x-3.9,2) [thick];
\draw[->] (2*\x-4,2) to (2*\x-7.9,2) [thick];}
\draw[->] (-4,0) to (-8,2) [thick];
\end{tikzpicture}
\end{center}
\end{exa}

\subsection{Semisimplicity}
In this section we prove that the category of finite dimensional modules for $MU_\vv(2)$ is semisimple, analogously to what happens with $\uvsl$. The strategy of the proof is also the same: we will use a mirabolic analogue of the Casimir element.
\begin{defin}\label{def:casimir}The \emph{mirabolic quantum Casimir element} is
 \begin{multline}C_{mir}:=(1-\vv^{-2})\left(fe+\frac{k\vv+k^{-1}\vv^{-1}}{(\vv-\vv^{-1})^2}\right)-fe\ell-\ell fe+\vv^2 f\ell e \\+\vv^{-2}e\ell f+(\vv^2-2)\ell \frac{k\vv+k^{-1}\vv^{-1}}{(\vv-\vv^{-1})^2}+\vv^{-2}\ell \frac{k\vv^{-1}+k^{-1}\vv}{(\vv-\vv^{-1})^2}.\end{multline}
\end{defin}
\begin{prop}We have $C_{mir}\in Z(MU_\vv(2))$.
\end{prop}
\begin{proof}We just need to check that $C_{mir}$ commutes with all the generators. The fact that $[C_{mir},k]=0$ (and hence $[C_{mir},k^{-1}]=0$) is immediate from the relations \eqref{rels:2}, \eqref{rels:3} and \eqref{rels:6}.
Now observe that 
\begin{align*}
\ell C_{mir}&=  (1-\vv^{-2})(\ell fe+\ell\frac{k\vv+k^{-1}\vv^{-1}}{(\vv-\vv^{-1})^2})-\ell fe\ell-\ell fe+\vv^2 f\ell e+ \\
&+ \vv^{-2}\ell e f+(\vv^2-2)\ell \frac{k\vv+k^{-1}\vv^{-1}}{(\vv-\vv^{-1})^2}+\vv^{-2}\ell \frac{k\vv^{-1}+k^{-1}\vv}{(\vv-\vv^{-1})^2}\\
&= -\vv^{-2}\ell f e-\ell fe\ell+\vv^2f\ell e+\vv^{-2}\ell\left(fe+\frac{k-k^{-1}}{\vv-\vv^{-1}}\right)\\
&+(\vv^2-1-\vv^{-2})\ell \frac{k\vv+k^{-1}\vv^{-1}}{(\vv-\vv^{-1})^2}+\vv^{-2}\ell \frac{k\vv^{-1}+k^{-1}\vv}{(\vv-\vv^{-1})^2}\\
&=-\ell fe\ell+\vv^2f\ell e+\vv^{-2}\frac{k-k^{-1}}{\vv-\vv^{-1}}+(\vv^2-1-\vv^{-2})\ell \frac{k\vv+k^{-1}\vv^{-1}}{(\vv-\vv^{-1})^2}+\vv^{-2}\ell \frac{k\vv^{-1}+k^{-1}\vv}{(\vv-\vv^{-1})^2}
\end{align*}
and
\begin{align*}
C_{mir}\ell &= (1-\vv^{-2})(fe\ell+\ell\frac{k\vv+k^{-1}\vv^{-1}}{(\vv-\vv^{-1})^2})- fe\ell-\ell fe\ell+\vv^2 f \ell e+ \\
& + \vv^{-2}e f\ell+(\vv^2-2)\ell\frac{k\vv+k^{-1}\vv^{-1}}{(\vv-\vv^{-1})^2}+\vv^{-2}\ell\frac{k\vv^{-1}+k^{-1}\vv}{(\vv-\vv^{-1})^2}\\
&=  -\vv^{-2}fe\ell-\ell fe\ell+\vv^2 f \ell e+ \vv^{-2}\left(fe+\frac{k-k^{-1}}{\vv-\vv^{-1}}\right)\ell+\\
&+(\vv^2-2)\ell\frac{k\vv+k^{-1}\vv^{-1}}{(\vv-\vv^{-1})^2}+\vv^{-2}\ell\frac{k\vv^{-1}+k^{-1}\vv}{(\vv-\vv^{-1})^2}\\
&=-\ell fe\ell+\vv^2f\ell e+\vv^{-2}\frac{k-k^{-1}}{\vv-\vv^{-1}}+(\vv^2-1-\vv^{-2})\ell \frac{k\vv+k^{-1}\vv^{-1}}{(\vv-\vv^{-1})^2}+\vv^{-2}\ell \frac{k\vv^{-1}+k^{-1}\vv}{(\vv-\vv^{-1})^2}
\end{align*}
hence $[C_{mir},\ell]=0$.

The proof that $eC_{mir}=C_{mir}e$ is a rather tedious computation that will be omitted, the strategy is to express everything in terms of the PBW basis of $MU_\vv(2)$. Finally, $fC_{mir}=C_{mir}f$ can be obtained from $eC_{mir}=C_{mir}e$ by using the antiautomorphism of Remark \ref{rem:anti-auto}.
\end{proof}
\begin{lem}The central element $C_{mir}$ acts by distinct scalars on the finite dimensional irreducible representations of $MU_\vv(2)$. More precisely, we have the following:
$$ C_{mir}|_{L^{\pm}(n,0)}=\pm\frac{\vv^n+\vv^{-n-2}}{\vv-\vv^{-1}};\qquad C_{mir}|_{L^{\pm}(n,1)}=\pm\frac{\vv^{n+2}+\vv^{-n}}{\vv-\vv^{-1}};\qquad C_{mir}|_{L^{\pm}(n,01)}=\pm\frac{\vv^n+\vv^{-n}}{\vv-\vv^{-1}}.$$
\end{lem}
\begin{proof}
First of all, remember that the usual quantum Casimir element for $\uvsl$ is 
$$C_\vv:=fe+\frac{k\vv+k^{-1}\vv^{-1}}{(\vv-\vv^{-1})^2}=ef+\frac{k\vv^{-1}+k^{-1}\vv}{(\vv-\vv^{-1})^2}$$
which acts as $\pm\frac{\vv^{n+1}+\vv^{-n-1}}{(\vv-\vv^{-1})^2}$ on $L^\pm(n)$.
Notice that if we set $\ell=0$, then $C_{mir}=(1-\vv^{-2})C_\vv$ and if we set $\ell=1$, then $C_{mir}=(\vv^2-1)C_\vv$. The first two equalities follow from this.

For the last equality, apply $C_{mir}$ to the highest weight vector $m^\pm_{0,0}\in L^\pm(n,01)$ of Prop. \ref{prop:L01} to obtain
\begin{align*}C_{mir}m^\pm_{0,0}&=(1-\vv^{-2})\left(fem^\pm_{0,0}+\frac{k\vv+k^{-1}\vv^{-1}}{(\vv-\vv^{-1})^2}m^\pm_{0,0}\right)-fe\ell m^\pm_{0,0}-\ell fe m^\pm_{0,0}+\vv^2 f\ell e m^\pm_{0,0} \\
&+\vv^{-2}e\ell f m^\pm_{0,0}+(\vv^2-2)\ell \frac{k\vv+k^{-1}\vv^{-1}}{(\vv-\vv^{-1})^2}m^\pm_{0,0}+\vv^{-2}\ell \frac{k\vv^{-1}+k^{-1}\vv}{(\vv-\vv^{-1})^2}m^\pm_{0,0} \\
&= \pm(1-\vv^{-2})\frac{\vv^{n+1}+\vv^{-n-1}}{(\vv-\vv^{-1})^2}m^\pm_{0,0}+\vv^{-2}e\ell(m^\pm_{1,0}+m^\pm_{1,1})\\
&=\pm \frac{\vv^n+\vv^{-n-2}}{\vv-\vv^{-1}}m^\pm_{0,0}+\vv^{-2}em^\pm_{1,1}\\
&=\pm\left(\frac{\vv^n+\vv^{-n-2}}{\vv-\vv^{-1}}+\vv^{-2}\vv^{1-n}\right)m^\pm_{0,0}\\
&=\pm\frac{\vv^n+\vv^{-n}}{\vv-\vv^{-1}}m^\pm_{0,0}.
\end{align*}
The result now follows because $C_{mir}$ has to act by the same scalar on the whole representation.
\end{proof}
\begin{theo}\label{thm:semisimple}Every finite dimensional $MU_\vv(2)$-module decomposes as a direct sum of simple modules.
\end{theo}
\begin{proof}The proof is exactly the same as the one for $\uvsl$, see for example \cite[2.9]{J}. Briefly, if $M$ is a finite dimensional $MU_\vv(2)$-module, it decomposes as a direct sum of generalized eigenspaces for the action of $C_{mir}$. Hence it is sufficient to show that each generalized eigenspace $M_{(\mu)}=\{m\in M~|~(C_{mir}-\mu)^km=0,\text{ for some }k\in\N\}$ is semisimple. Reduce then to the case where $M=M_{(\mu)}$. In this case, $M$ has a filtration $\{M_i\}_{i=0}^r$ such that $M_i/M_{i-1}\simeq L$ for some simple $MU_\vv(2)$-module $L$, since $C_{mir}$ has to act by the same scalar on each $M_i/M_{i-1}$. Considering the weight space decomposition $M=\bigoplus M_{\lambda,\epsilon}$, we have that $\dim M_{\lambda,\epsilon}=r\dim L_{\lambda,\epsilon}$. Let $(\lambda_0,\epsilon_0)$ be the highest weight of $L$, and pick a basis $v_1,\ldots, v_r$ of $M_{\lambda_0,\epsilon_0}$. By the proof of Theorem \ref{thm:classification}, it is clear that $MU_\vv(2)v_i\simeq L$ for each $i=1,\ldots,r$. We have that $M\simeq\sum_{i=1}^r MU_\vv(2)v_i$ and by counting the dimensions of the weight spaces the sum has to be direct, which gives the result. 
\end{proof}
\begin{cor}\label{cor:decomp}Let $M$, $N$ be finite dimensional $MU_\vv(2)$-modules with the same weight space decomposition, i.e. $\dim(M_{\lambda,\epsilon})=\dim (N_{\lambda,\epsilon})$ for all $\lambda\in\pm\vv^{\Z}$, $\epsilon\in\{0,1\}$. Then $M\simeq N$ as $MU_\vv(2)$-modules.
\end{cor}
\begin{proof}Both $M$ and $N$ decompose as a direct sum of irreducibles by Theorem \ref{thm:semisimple}, hence the result follows from inspecting the weight space decomposition of the irreducibles and observing that the weight spaces of a sum of irreducibles of type $L^\pm(n,01)$ can be never be equal to the weight spaces of a sum of modules of the types $L^\pm(n,1)$ and $L^\pm(n,0)$.
\end{proof}
\section{Mirabolic Schur-Weyl duality}\label{sec:sw}
The goal of this section is to describe a natural Schur-Weyl type duality between the mirabolic Hecke algebra $\mathcal{R}_d(q)$ of \cite{R14} and the mirabolic quantum Schur algebra $\mathcal{MU}_q(n,d)$. This is done from the point of view of convolution algebras on flag varieties, similarly to the interpretation by Grojnowski and Lusztig in \cite{GL} of the quantum version of Schur-Weyl duality due to Jimbo.

As in Section \ref{sec:conv-mir}, we denore by $\F_q$ the finite field with $q$ elements and $G_d:=\GL_d(\F_q)$ for all $d\in\N_+$. Let $B_d\subseteq G_d$ be the Borel subgroup of upper triangular invertible matrices, then we have the variety of complete flags in $\F_q^d$:
$$ G_d/B_d\simeq \{F=(0=F_0\subseteq F_1\subseteq\ldots\subseteq F_d=\F_q^d)~|~\dim F_i=i,~i=1,\ldots,d\}.$$
As defined in \cite[$\S$ 3]{R14} the algebra $\mathcal{R}_d(q)=\C(G_d/B_d\times G_d/B_d\times \F_q^d)^{G_d}$, with the same convolution product as in \eqref{eq:convolution}, is called the \emph{mirabolic Hecke algebra}. 

Consider the space $\cF(n,d)\times G_d/B_d\times \F_q^d$ of triples of one $n$-step flag, one complete flag and a vector in $\F_q^d$. The group $G_d$ acts diagonally on $\cF(n,d)\times G_d/B_d\times \F_q^d$ with finitely many orbits. These orbits can be parametrized in an analogous way to what we did in $\S$ \ref{subsec:orbits} in terms of decorated matrices (see \cite{MWZ}). Let 
$$\Theta_{n,1^d}:=\{A=(a_{ij})\in M_{n\times d}(\N)~|~\co(A)=(1^d)\}$$
and
$$\Xi_{n,1^d}:=\{(A,\Delta) ~|~A\in\Theta_{n,1^d}, ~\Delta\text{ as in  Def. \ref{def:dec-mat} }\}.$$
Then we have a bijection
$$ G_d\backslash \left(\cF(n,d)\times G_d/B_d \times \F_q^d\right) \longleftrightarrow \Xi_{n,1^d}.$$
\begin{rem}\label{rem:mat-to-seq}
We can represent matrices in $\Theta_{n,1^d}$ also as elements of $\{1,\ldots,n\}^d$. 
For a given $A=(a_{ij})\in\Theta_{n,1^d}$, we define a sequence $\underline{i}(A)=(i_1,\ldots,i_d)\in\{1,\ldots,n\}^d$ by setting $i_r=m$ if $a_{mr}=1$. This is well defined because there is only one entry equal to $1$ in each column of $A$ and all the other entries are zero. This clearly gives a bijection.
Pairs $(A,\Delta)\in\Xi_{n,1^d}$ can then be represented as pairs $\{\underline{i},J\}$ where $J=\{j_1,\ldots,j_k\}\subseteq \{1,\ldots,d\}$ is such that $i_{j_1}>i_{j_2}>\ldots> i_{j_k}$. In this case the bijection is given by defining $\underline{i}=\underline{i}(A)$ as above and $J=\{j_1,\ldots,j_k\}$ if $\Delta=\{(i_1,j_1),\ldots,(i_k,j_k)\}$.
\end{rem}
\begin{exa}Let $d=5$, $n=3$, the following decorated matrix (using the circle notation of Example \ref{ex:M1})
$$ \begin{pmatrix}0 & 0 & 0 & \enc{1} & 0 \\ 1 & 0 & 0 & 0 & 1 \\ 0 & \enc{1} & 1 & 0 & 0 
\end{pmatrix} $$
corresponds to the pair $(23312,\{2,4\})$.
\end{exa}
\begin{lem}\label{lem:count-Xi}The number of orbits of $G_d$ on $\cF(n,d)\times G_d/B_d\times \F_q^d$ is
$$ | \Xi_{n,1^d}|=\sum_{k=0}^{\min\{n,d\}} {d \choose k} {n \choose k}n^{d-k}.$$
\end{lem}
\begin{proof}We count pairs $(\underline{i},J)$ as in Remark \ref{rem:mat-to-seq}. Let $X_k=\{(\underline{i},J)\in \Xi_{n,1^d}~|~|J|=k\}$. Clearly $X_k\neq\emptyset$ if and only if $0\leq k\leq \min\{n,d\}$.
To count elements in $X_k$, first consider that there are ${d \choose k}$ possibilities for what $J$ can be and, for each of those, the elements $i_r$, $r\in J$ are determined simply by the choice of $k$ elements in $\{1,\ldots,n\}$, by the decreasing condition. Finally, the sequence elements $i_r$, $r\not\in J$ can be anything in $\{1,\ldots,n\}^{d-k}$. The result follows then from the fact that $\Xi_{n,1^d}=\sqcup_{k=0}^{\min\{n,d\}}X_k$.
\end{proof}
\begin{defin}\label{def:tensor-space}We define 
$\mathcal{MT}_q(n,d):=\C(\cF(n,d)\times G_d/B_d \times \F_q^d)^{G_d},$
which has a left action by $\mathcal{MU}_q(n,d)$ and a right action by $\mathcal{R}_d(q)$ defined as in \eqref{eq:convolution}. 

More explicitly, for $\alpha\in \C(\cF(n,d)\times \cF(n,d)\times \F_q^d)^{G_d}$, $\beta\in \C(G_d/B_d\times G_d/B_d\times \F_q^d)^{G_d}$, $\gamma\in\C(\cF(n,d)\times G_d/B_d\times \F_q^d)^{G_d}$, 
 we have:
\begin{align*}(\alpha * \gamma) (F,F',v)&:=\sum_{H\in\cF(n,d),~u\in\F_q^d}\alpha(F,H,u)\gamma(H,F',v-u);\\
(\gamma*\beta)(F,F',v)&:=\sum_{H\in G_d/B_d,~u\in\F_q^d}\gamma(F,H,u)\beta(H,F',v-u).
\end{align*}
In analogy with the non-mirabolic case, we call $\mathcal{MT}_q(n,d)$ the \emph{mirabolic tensor space}, even though it is not a tensor product. 
\end{defin}
\begin{rem}
Exactly in the same way as for $\mathcal{MU}_q(n,d)$, $\mathcal{MT}_q(n,d)$ has a basis $\{T_{A,\Delta}~|~(A,\Delta)\in\Xi_{n,1^d}\}$ where $T_{A,\Delta}$ is the characteristic function of the orbit $\mathcal{O}_{A,\Delta}$.
\end{rem}
\begin{lem}\label{lem:commutant}For the $(\mathcal{MU}_q(n,d),\mathcal{R}_d(q))$-bimodule $\mathcal{MT}_q(n,d)$ we have
$$ \End_{ \mathcal{R}_d(q)}(\mathcal{MT}_q(n,d))\simeq\mathcal{MU}_q(n,d).$$
\end{lem}
\begin{proof}
Let $P=G_d\ltimes \F_q^d$ be the group of affine transformations of $\F_q^d$ and, for any composition of $d$ with $n$-parts $\mu=(\mu_1,\ldots,\mu_n)$  (i.e. $\mu_i\in\N$, $i=1,\ldots n$, $\sum_i \mu_i=d$), we let $P^\mu$ be the parabolic subgroup of $G_d$ consisting of block upper triangular matrices with blocks of sizes $(\mu_1,\ldots,\mu_n)$. If we let $\mathcal{C}_{n,d}$ be the set of all such compositions, then
$ \cF(n,d)\simeq \bigsqcup_{\mu\in\mathcal{C}_{n,d}}G_d/P^\mu.$

With the natural inclusions $B_d\subseteq P$ and $P^\mu\subseteq P$, we define the idempotents corresponding to the subgroups 
$$e_B:=\frac{1}{|B_d|}\sum_{b\in B_d}b\in\C[P]\quad\text{ and }\quad e_{\mu}:=\frac{1}{|P^\mu|}\sum_{p\in P^\mu}p\in\C[P].$$
Notice that $e_Be_{\mu}=e_\mu e_B=e_\mu$ for all $\mu\in\mathcal{C}_{n,d}$ because $B_d\subseteq P^\mu$.
Then, as proved in \cite[$\S$ 3.1]{R14}, we have that $\mathcal{R}_d(q)\simeq \End_{\C[P]}(\C[P]e_B)=e_B\C[P] e_B$ and in exactly the same way it can be shown that 
$$\mathcal{MU}_q(n,d)\simeq \End_{\C[P]}\left(\bigoplus_{\mu\in\mathcal{C}_{n,d}}\C[P]e_{\mu}\right)=\bigoplus_{\mu,\nu\in\mathcal{C}_{n,d}}e_\mu\C[P]e_\nu;$$
$$ \mathcal{MT}_q(n,d)\simeq\bigoplus_{\mu\in\mathcal{C}_{n,d}}e_\mu\C[P]e_B.$$
The result now follows from the fact that for all $\mu,\nu\in\mathcal{C}_{n,d}$ we have
$$\Hom_{e_B\C[P]e_B}(e_\mu\C[P]e_B,e_\nu\C[P]e_B)\simeq e_\nu (e_B \C[P]e_B)e_\mu=e_\nu\C[P] e_\mu.$$
\end{proof}

\begin{rem}The structure constants of the actions in Def.~ \ref{def:tensor-space} are polynomials in $\Z[q]$, hence we can argue as in Def.~\ref{def:extend-scalars} and consider $\mathcal{MT}_q(n,d)$ to be the specialization at $\q\mapsto q$ of a certain $\C[\q,\q^{-1}]$-module with a left action by $\mathcal{MU}_\q(n,d)$ and a right action by the (non-specialized) mirabolic Hecke algebra. We can then extend scalars to $\C(\vv)$, where $\vv^2=\q$ and we denote the resulting generic mirabolic tensor space by $\mathrm{MT}_\vv(n,d)$ and the generic mirabolic Hecke algebra by $R_d$ (notice that this notation differs from \cite[Def.~3.2]{R14} and in that paper the square root of $\q$ was never introduced). In what follows we use these generic version of the algebras, but the same results hold for any of the semisimple specializations.
\end{rem}
Since $R_d$ is a semisimple algebra and, by Lemma \ref{lem:commutant}, $\End_{R_d}(\mathrm{MT}_\vv(n,d))=MU_\vv(n,d)$, the double commutant theorem tells us also that the image of $R_d$ in $\End(\mathrm{MT}_\vv(n,d))$ centralizes the action of $MU_\vv(n,d) $
and that we have a decomposition
\begin{equation}\label{eq:sw-decomp}\mathrm{MT}_\vv(n,d)\simeq \bigoplus_{\lambda\in\Lambda}L_\lambda\otimes V_\lambda
\end{equation}
where $L_\lambda$ and $V_\lambda$ are non-isomorphic simple modules for $MU_\vv(n,d)$ and $R_d$ respectively and $\lambda$ runs over a certain finite index set $\Lambda$.
\subsection{The case of $\mathrm{MT}_\vv(2,d)$} Since $MU_\vv(2,d)$ is a quotient of $MU_\vv(2)$, of which we have classified the irreducible representations in Theorem \ref{thm:classification}, we can be more explicit about the decomposition \eqref{eq:sw-decomp} in the case when $n=2$.

Remember that the usual quantum Schur-Weyl duality says that $(\C(\vv)^2)^{\otimes d}$ decomposes, as a bimodule for $\uvsl$ and the Hecke algebra $H_d$, as 
\begin{equation}\label{eq:sw2}(\C(\vv)^2)^{\otimes d}\simeq \bigoplus_{\substack{\lambda\vdash d \\ \lambda=(\lambda_1,\lambda_2)}}L^+(\lambda_1-\lambda_2)\otimes S_{\lambda}
\end{equation}
where $\lambda_2$ can be equal to zero and $S_\lambda$ is the irreducible representation of $H_d$ corresponding to the partition $\lambda$. Remember that $\dim S_\lambda= f_\lambda$, which is the number of standard Young tableaux of shape $\lambda$. We want a mirabolic analogue of this decomposition.
Recall from \cite[$\S$3.2]{R14} that the mirabolic Hecke algebra $R_d$ is a semisimple algebra and its irreducible representations can be written as $M^{(\lambda,1^s)}$ where $(\lambda,1^s)$ is a bipartition of $d$. Also $\dim M^{(\lambda,1^s)}={d\choose s}f_\lambda$. We can then conjecture the mirabolic analogue of \eqref{eq:sw2} to be as follows.
\begin{conj}\label{conj:sw}The decomposition of \eqref{eq:sw-decomp} in the case $n=2$ becomes
\begin{equation}\label{eq:sw2-mir}\mathrm{MT}_\vv(2,d)\simeq\bigoplus_{\blambda\in\Lambda}L^+_{\blambda}\otimes M^{\blambda}.
\end{equation}
Here $\blambda$ runs over the set $\Lambda=\{ (\lambda,1^s)~|~|\lambda|+s=d,~\lambda=(\lambda_1,\lambda_2),~0\leq s\leq 2\}$ of bipartitions of $d$ where each partition has at most two parts and the second partition is a single column; $M^{\blambda}$ is the irreducible representation of $R_d$ corresponding to the bipartition $\blambda$, and
$$ L^+_{\blambda}=\begin{cases} L^+(\lambda_1-\lambda_2,1) & \mbox{ if }  \blambda=(\lambda,\emptyset) \\
L^+(\lambda_1-\lambda_2+1,01) & \mbox{ if } \blambda=(\lambda,1) \\
L^+(\lambda_1-\lambda_2,0) & \mbox{ if }\blambda=(\lambda,11) \end{cases}.$$
\end{conj}

This conjecture was verified by direct computation for $d=1,2,3$. In fact, working out this decomposition for $d=3$ led to identifying the patterns involved in the classification of the irreducible representations of $MU_\vv(2)$.

Some of the features of the usual Schur-Weyl duality are missing here, namely the fact that the mirabolic tensor space is not actually a tensor product, in fact it is not even clear whether $MU_\vv(2)$ can be made into a bialgebra. However, we can still say something about the structure of $\mathrm{MT}_\vv(2,d)$ as a left $MU_\vv(2)$-module.
\begin{theo}\label{thm:sw-MU2}The isomorphism \eqref{eq:sw2-mir} holds as a map of left $MU_\vv(2)$-modules.
\end{theo}
\begin{proof}By Corollary \ref{cor:decomp}, it is enough to check that both sides have the same multiplicity of weight spaces. 
Remember that the weight space decomposition for $(\C(\vv)^2)^{\otimes d}$ as a $\uvsl$-module is given by binomial coefficients, i.e. 
$$\dim \left((\C(\vv)^2)^{\otimes d}\right)_{\vv^{d-2r}}=\dps{d\choose r}\text{ for all }r=0,1,\ldots,d.$$ 
Now the right hand side of \eqref{eq:sw2-mir} is equal to $ T^{(1)}\oplus T^{(01)}\oplus T^{(0)}$
where
\begin{align*} T^{(1)}&\simeq \bigoplus_{(\lambda,\emptyset)\in\Lambda}L^+(\lambda_1-\lambda_2,1)\otimes M^{(\lambda,\emptyset)}, \\
T^{(01)} &\simeq \bigoplus_{(\lambda,1)\in\Lambda}L^+(\lambda_1-\lambda_2+1,01)\otimes M^{(\lambda,1)},\\
T^{(0)}&\simeq  \bigoplus_{(\lambda,11)\in\Lambda}L^+(\lambda_1-\lambda_2,0)\otimes M^{(\lambda,11)}.
\end{align*}
We have isomorphisms of $MU_\vv(2)$-modules
\begin{equation*}\bigoplus_{(\lambda,\emptyset)\in\Lambda}L^+(\lambda_1-\lambda_2,1)\otimes M^{(\lambda,\emptyset)}\simeq \bigoplus_{\substack{\lambda\vdash d \\ \lambda=(\lambda_1,\lambda_2)}}\left(L^+(\lambda_1-\lambda_2,1)\right)^{\oplus f_\lambda}
\stackrel{\text{ by }\eqref{eq:sw2}}{\simeq} \pi_1^* (\C(\vv)^2)^{\otimes d};
\end{equation*}
\begin{equation*}
\bigoplus_{(\lambda,11)\in\Lambda}L^+(\lambda_1-\lambda_2,0)\otimes M^{(\lambda,11)}\simeq \bigoplus_{\substack{\lambda\vdash d-2 \\ \lambda=(\lambda_1,\lambda_2)}}L^+(\lambda_1-\lambda_2,0)^{\oplus{d\choose 2} f_\lambda}\stackrel{\text{ by }\eqref{eq:sw2}}{\simeq} \pi_0^* \left((\C(\vv)^2)^{\otimes (d-2)}\right)^{\oplus {d\choose 2}}.
\end{equation*}
Hence  
\begin{equation}\label{eq:wt-1-0} \dim \left(T^{(1)}\right)_{\vv^{d-2r},\epsilon}=\begin{cases} {d\choose r} & \mbox{ if } \epsilon=1 \\ 0 &\mbox{ if }\epsilon=0 \end{cases};\qquad \dim \left(T^{(0)}\right)_{\vv^{d-2r},\epsilon}=\begin{cases} 0 & \mbox{ if } \epsilon=1 \\ {d\choose 2}{d-2 \choose r-1} &\mbox{ if }\epsilon=0 \end{cases}.\end{equation}
We also have, for all $k$,
$$ \dim\left(L^+(\lambda_1-\lambda_2+1,01)\right)_{\vv^{d-2r},\epsilon}=\begin{cases} \dim\left(L^+(\lambda_1-\lambda_2,1)\right)_{\vv^{d+1-2r},1} & \mbox{ if }\epsilon=1 \\
\dim\left(L^+(\lambda_1-\lambda_2,0)\right)_{\vv^{d-1-2r},0} & \mbox{ if }\epsilon=0 \end{cases}.$$
Since
$$ \bigoplus_{(\lambda,1)\in\Lambda}L^+(\lambda_1-\lambda_2+1,01)\otimes M^{(\lambda,1)}\simeq  \bigoplus_{\substack{\lambda\vdash d-1 \\ \lambda=(\lambda_1,\lambda_2)}}L^+(\lambda_1-\lambda_2+1,01)^{\oplus d f_\lambda}$$
as $MU_\vv(2)$-modules, we have
\begin{equation}\label{eq:wt-01} \dim \left(T^{(01)}\right)_{\vv^{d-2r},\epsilon}=\begin{cases} d{d-1\choose r-1} & \mbox{ if } \epsilon=1 \\  d{d-1\choose r}&\mbox{ if }\epsilon=0 \end{cases}.\end{equation}
By \eqref{eq:wt-1-0} and \eqref{eq:wt-01} we can conclude that the weight multiplicities of the right hand side of \eqref{eq:sw2-mir} are equal to
\begin{equation}\label{eq:wt-RHS}\dim\left(T^{(1)}+T^{(01)}+T^{(0)}\right)_{\vv^{d-2r},\epsilon}=\begin{cases}{d\choose r}+ d{d-1\choose r-1} & \mbox{ if } \epsilon=1 \\  d{d-1\choose r}+{d\choose 2}{d-2\choose r-1}&\mbox{ if }\epsilon=0 \end{cases}.
\end{equation}
To compute the weight space multiplicities of the left hand side we need to look at the action of $k,\ell\in MU_\vv(2)$ on the basis elements $\{T_{A,\Delta}~|~(A,\Delta)\in \Xi_{2,1^d}\}$ of $\mathrm{MT}_\vv(2,d)$. For simplicity of notation, we will actually identify the pairs $(A,\Delta)$ with pairs $(\underline{i},J)$ as in Remark \ref{rem:mat-to-seq} and write $T_{\underline{i},J}$ for the corresponding basis element. Notice that, for a fixed $\ul{i}\in\{1,2\}^d$, the possibilities for $J$ such that $(\ul{i},J)\in\Xi_{2,1^d}$ are as follows: either $J=\emptyset$, or $J=\{j\}$ for any $j\in\{1,\ldots,d\}$, or $J=\{j,m\}$  for any $j,m\in\{1,\ldots,d\}$ such that $j<m$ and $i_j>i_m$. 

It is immediate from the definition of the action that 
$$1_r\cdot T_{\underline{i},J}=\delta_{r,\underline{i}(1)}T_{\underline{i},J},\text{ where }\underline{i}(1)=\#\{p\in\{1,\ldots,d\}~|~i_p=1\},$$
hence
\begin{equation}k\cdot T_{\underline{i},J}=\vv^{2\underline{i}(1)-d}T_{\underline{i},J}=\vv^{d-2\underline{i}(2)}T_{\underline{i},J},\text{ where }\ul{i}(2)=d-\ul{i}(1)=\#\{p\in\{1,\ldots,d\}~|~i_p=2\}.
\end{equation}
The action of $x_r$ on $\mathrm{MT}_\vv(2,d)$ can also be readily computed in terms of counting flags and vectors (details are omitted but are similar to the arguments in the proof of Prop.~\ref{prop:comp-e-action}). This then gives us that
\begin{equation}\label{eq:ell-act}
\ell \cdot T_{\underline{i},J}=\begin{cases} \vv^{-2\underline{i}(1)}\left(T_{\underline{i},\emptyset}+\dps\sum_{\substack{j\in[1,d]\\i_j=1}}T_{\underline{i},\{j\}}\right) & \mbox{ if }J=\emptyset;
\\  \vv^{-2\underline{i}(1)+2\varphi_j}(\vv^2-1)\left(T_{\underline{i},\emptyset}+\dps\sum_{\substack{j\in[1,d]\\i_j=1}}T_{\underline{i},\{j\}}\right) & \mbox{ if }J=\{j\},~ i_j=1; \\
\vv^{-2\underline{i}(1)+2\varphi_j}\left(T_{\underline{i},\{j\}}+\dps\sum_{\substack{m>j\\i_{m}=1}}T_{\underline{i},\{j,m\}}\right) & \mbox{ if }J=\{j\},~ i_j=2; \\
\vv^{-2\underline{i}(1)+2\varphi_m}(\vv^2-1)\left(T_{\underline{i},\{j\}}+\dps\sum_{\substack{m'>j\\i_{m'}=1}}T_{\underline{i},\{j,m'\}}\right) & \mbox{ if }J=\{j,m\};
\end{cases}
\end{equation}
where $\varphi_j=\#\{p<j~|~i_p=1\}$.

For a fixed $\underline{i}\in\{1,2\}^d$, consider the subspace $V_{\underline{i}}=\Span\{T_{\underline{i},J}~|~J\}\subseteq \mathrm{MT}_\vv(2,d)$. 

Clearly $\mathrm{MT}_\vv(2,d)=\bigoplus_{\ul{i}\in\{1,2\}^d}V_{\ul{i}}$ and $V_{\underline{i}}$ is invariant under the action of $k$ (which acts as the constant $\vv^{d-2(d-\underline{i}(1))}$) and, by \eqref{eq:ell-act}, under the action of $\ell$. 

Let $\inv(\ul{i})=\#\{(j,m)\in \{1,\ldots,d\}^2~|~j<m,~i_j>i_m\}$. Then $\dim V_{\ul{i}}=d+1+\inv(\ul{i})$. From \eqref{eq:ell-act}, it follows that 
$$\dim (\im\ell|_{V_{\ul{i}}})=1+\ul{i}(2)$$
therefore
$$ \dim(\ker\ell|_{V_{\ul{i}}})=d+1+\inv(\ul{i})-\dim(\im\ell|_{V_{\ul{i}}})=d-\ul{i}(2)+\inv(\ul{i}).$$
We can now compute the weight decomposition for $\mathrm{MT}_\vv(2,d)$.
\begin{align*} \dim\left(\mathrm{MT}_\vv(2,d)\right)_{\vv^{d-2r},1}&=\sum_{\substack{\ul{i}\in\{1,2\}^d \\ \ul{i}(2)=r}}\dim(\im\ell|_{V_{\ul{i}}}) =\sum_{\substack{\ul{i}\in\{1,2\}^d \\ \ul{i}(2)=r}}(1+\ul{i}(2)) \\
&=\sum_{\substack{\ul{i}\in\{1,2\}^d \\ \ul{i}(2)=r}}(1+r)=(1+r){d \choose r} \\
&= {d \choose r}+r{d \choose r} ={d \choose r}+d{d-1\choose r-1}
\end{align*}
which agrees with the case $\epsilon=1$ of \eqref{eq:wt-RHS}.

To conclude, first of all note that for $r=0$ or $r=d$ there is only one $\ul{i}\in\{1,2\}^d$ such that $\ul{i}(2)=r$ and in this case $\inv(\ul{i})=0$. Then, we will show that, for all $1\leq r\leq d-1$, we have 
\begin{equation}\label{eq:count-inv}\sum_{\substack{\ul{i}\in\{1,2\}^d \\ \ul{i}(2)=r}}\inv(\ul{i})={r+1 \choose 2}{d \choose r+1}={d\choose 2}{d-2\choose r-1}.\end{equation}
We induct on $d$. For $d=2$, $r=1$ so we have $\{\ul{i}\in\{1,2\}^2~|~\ul{i}(2)=1\}=\{(12),(21)\}$ and $\inv(12)+\inv(21)=0+1=1={2 \choose 2}{2\choose 2}$.

Now suppose $d\geq 3$ and let $Y_{d,r}=\{\ul{i}\in\{1,2\}^d~|~\ul{i}(2)=r\}$. We have $Y_{d,r}=Y^1_{d,r}\sqcup Y^2_{d,r}$, where $Y^c_{d,r}=\{\ul{i}\in Y_{d,r}~|~i_d=c\}$ for $c=1,2$. For $\ul{i}\in\{1,2\}^d$, let $\ul{i}'=(i_1\ldots i_{d-1})\in\{1,2\}^{d-1}$. Then it is immediate that if $\ul{i}\in Y^1_{d,r}$, then $\inv(\ul{i})=\inv(\ul{i}')+\ul{i}(2)$. Also, if $\ul{i}\in Y^2_{d,r}$, then $\inv(\ul{i})=\inv(\ul{i}')$. Hence
\begin{align*}\sum_{\ul{i}\in Y_{r,d}}\inv(\ul{i})&=\sum_{\ul{i}\in Y^1_{r,d}}\inv(\ul{i})+\sum_{\ul{i}\in Y^2_{r,d}}\inv(\ul{i}) \\
&= \sum_{\ul{i}'\in Y_{r,d-1}}(\inv(\ul{i}')+r)+\sum_{\ul{i}'\in Y_{r-1,d-1}}\inv(\ul{i}') \\
( \text{ by ind. hyp. })&={r+1\choose 2}{d-1\choose r+1}+r{d-1\choose r}+{r\choose 2}{d-1\choose r} \\
&= \frac{(d-1)!}{(r-1)!}\left(\frac{1}{2(d-r-2)!}+\frac{1}{(d-r-1)!}+\frac{r-1}{2(d-r-1)!}\right)\\
&= \frac{d!}{2(r-1)!(d-r-1)!}={r+1\choose 2}{d \choose r+1}
\end{align*}
which proves \eqref{eq:count-inv}.

Finally
\begin{align*} \dim\left(\mathrm{MT}_\vv(2,d)\right)_{\vv^{d-2r},0}&=\sum_{\substack{\ul{i}\in\{1,2\}^d \\ \ul{i}(2)=r}}\dim(\ker\ell|_{V_{\ul{i}}}) =\sum_{\substack{\ul{i}\in\{1,2\}^d \\ \ul{i}(2)=r}}(d-\ul{i}(2)+\inv(\ul{i})) \\
&= \sum_{\substack{\ul{i}\in\{1,2\}^d \\ \ul{i}(2)=r}}(d-r)+\sum_{\substack{\ul{i}\in\{1,2\}^d \\ \ul{i}(1)=r}}\inv(\ul{i})\stackrel{\eqref{eq:count-inv}}{=}(d-r){d\choose r}+{r+1\choose 2}{d \choose r+1} \\
&=d{d-1\choose r}+{d\choose 2}{d-2\choose r-1}
\end{align*}
which is the same as the case $\epsilon=0$ in \eqref{eq:wt-RHS}.
\end{proof}
\begin{rem}Theorem \ref{thm:sw-MU2} in particular implies that each finite dimensional simple representation of $MU_\vv(2)$ with $k$-eigenvalues in $\vv^\Z$ appears as a summand of a mirabolic tensor space. In particular, $L^+(n,1)$ and $L^+(n,01)$ are summands of $\mathrm{MT}_\vv(2,n)$, while $L^+(n,0)$ is a summand of $\mathrm{MT}_\vv(2,n+2)$.
\end{rem}
\begin{rem}Theorem \ref{thm:sw-MU2} almost proves Conjecture \ref{conj:sw}, because it tells us that the decomposition \eqref{eq:sw2-mir} has to be true for some simple $R_d$-modules of the correct dimensions. Unfortunately, dimension alone is not enough to identify the modules uniquely. One possible strategy would be to find the eigenvalues for the action of the Jucys-Murphy elements, described in \cite[\S 6]{R14}.
\end{rem}

\bibliographystyle{alpha}
\bibliography{Rosso-biblist} 
\end{document}